\documentclass[a4paper,12pt]{amsart}
\usepackage[utf8]{inputenc}
\usepackage[T1]{fontenc}
\usepackage[UKenglish]{babel}
\usepackage[margin=18mm]{geometry}

\usepackage{color}

\usepackage{graphicx}

\usepackage{amsmath,amssymb,amsfonts,amsthm}
\usepackage{mathrsfs,eucal,dsfont}
\usepackage{verbatim,enumitem}

\usepackage{hyperref,url}

\def\scr#1{\mathscr{#1}}
\newcommand{\et}{\quad\text{and}\quad}

\newcommand{\R}{\mathds R}

\newcommand{\I}{\mathds 1}

\def\d{{\rm d}}
\def\<{\langle}
\def\>{\rangle}
\def\wt{\widetilde}
\def\ol{\overline}

 \def\tt{\tilde}
 
 \def\ss{\sqrt}

\def\R{\mathbb R}   \def\ss{\sqrt} 
  
  \def\vv{\varepsilon} 
\def\<{\langle} \def\>{\rangle}  
  \def\nn{\nabla}  
\def\d{\text{\rm{d}}}   
  \def\si{\sigma} 
 \def\beq{\begin{equation}}  
 
\def\e{\text{\rm{e}}}  \def\OO{\Omega}  
 \def\tt{\tilde} 
 \def\P{\mathbb P}

\def\E{\mathbb E}

\def\to{\rightarrow}
\def\8{\infty}\def\3{\triangle}
\def\W{\mathbb{W}}\def\1{\lesssim}

\renewcommand{\bar}{\overline}
\renewcommand{\hat}{\widehat}
\renewcommand{\tilde}{\widetilde}

\newtheorem{theorem}{Theorem}[section]
\newtheorem{lemma}[theorem]{Lemma}
\newtheorem{proposition}[theorem]{Proposition}

\theoremstyle{definition}

\newtheorem{example}[theorem]{Example}
\newtheorem{remark}[theorem]{Remark}

\numberwithin{equation}{section}
\begin{document}
\allowdisplaybreaks

\title[Conditional McKean-Vlasov jump diffusions] {Ergodicity of conditional McKean-Vlasov jump diffusions}

\author{Jianhai Bao \qquad
Yao Liu\qquad
Jian Wang}
\date{}

\thanks{\emph{J.\ Bao:} Center for Applied Mathematics, Tianjin University, 300072  Tianjin, P.R. China. \url{jianhaibao@tju.edu.cn}}

\thanks{\emph{Y.\ Liu:} School  of Mathematics and Statistics, Fujian Normal University, 350007 Fuzhou, P.R. China. \url{liuyaomath@163.com}}

\thanks{\emph{J.\ Wang:}
School  of Mathematics and Statistics \& Key Laboratory of Analytical Mathematics and Applications (Ministry of Education) \& Fujian Provincial Key Laboratory
of Statistics and Artificial Intelligence, Fujian Normal University, 350007 Fuzhou, P.R. China. \url{jianwang@fjnu.edu.cn}}

\maketitle

\begin{abstract}
In this paper, we are interested in   conditional McKean-Vlasov jump diffusions, which are also termed as  McKean-Vlasov stochastic differential equations with jump idiosyncratic  noise and jump
common noise. As far as conditional McKean-Vlasov jump diffusions are concerned,  the corresponding conditional distribution flow is a measure-valued process, which indeed  satisfies a stochastic partial integral differential equation driven by a Poisson random measure.
Via a novel  construction of the  asymptotic coupling by reflection,  we explore  the ergodicity  of the underlying measure-valued process  corresponding to  a
  one-dimensional conditional McKean-Vlasov jump diffusion when the associated drift term fulfils a partially dissipative condition with respect to the spatial variable. In addition, the theory  derived demonstrates that the intensity of
 the jump common noise and the jump idiosyncratic noise can simultaneously
 enhance
  the convergence rate of the exponential ergodicity.

\medskip


\smallskip

\end{abstract}
\section{Introduction and main result}
\subsection{Background}

When the coefficients of an SDE under consideration depend not only on the state of the solution but also the   law of the solution itself, it is
referred to as a distribution-dependent SDE \cite{WR}. In
the literature, the distribution-dependent SDE is also termed as a  mean-field SDE \cite{CD1} or a McKean-Vlasov SDE  in honor of the mean-field concept in kinetic theory \cite{Kac} due to Vlasov and establishing an SDE framework which links particle systems to nonlinear diffusions \cite{MH1,Sznitman}. In the past few decades, McKean-Vlasov SDEs have been applied considerably \cite{CD1} in statistical physics, mean-field games, finance, and collective behavior modeling, to name just a few. In contrast to the classical SDEs,  due to the nonlinear dependence on the measure variables,   some challenges need to be  surmounted in order  to tackle the finite-time behavior  and the long-time asymptotics  for McKean-Vlasov SDEs. In particular,
 the issues on
strong/weak well-posedness,
stochastic
numerics, propagation of chaos (PoC for short), ergodicity as well as existence and uniqueness of stationary distributions
have advanced greatly; see, for instance, \cite{BLW,CD1,Cava,dES,GLM,LMW,Schuh,Wanga}.

Admittedly, a McKean-Vlasov SDE  builds a bridge between microscopic interactions and macroscopic phenomena. Nevertheless, the McKean-Vlasov SDE is incompetent to depict the systemic randomness (which influences   all particles concurrently) in an interconnected system. In turn,
the McKean-Vlasov SDE with common noise plays a proper role in modelling  a complex system, which  enjoys an emergent structure and is subject to shared shocks. In  terminology, the McKean-Vlasov SDE with common noise is also called the conditional McKean-Vlasov SDE; see, for example,  \cite{BLM,Huang,LSZ,Pham,STW}. The distinctions between standard McKean-Vlasov SDEs and conditional McKean-Vlasov SDEs lie  in  measure dependence (deterministic vs stochastic), particle independence (independent vs ‌conditionally independent at infinity), and nonlinear Fokker-Planck equations (PDE vs SPDE), and so forth. Regarding conditional McKean-Vlasov SDEs,
the   discrepancies mentioned previously might (partially)  lead  to invalidity of the existing methods dealing with standard McKean-Vlasov SDEs.
With  wide applications in e.g.  mean-field games  with partial information \cite{CD2}, nonlinear filtering problems, stochastic control with partial observation and  mean-field interactions, systemic risk modeling in finance \cite{BLY}, conditional McKean-Vlasov SDEs driven by Brownian motions
have been explored in depth upon ergodicity \cite{BW,DTM,Maillet}, well-posedness \cite{BLM,HSSb,KNRS},  conditional  PoC  \cite{ELL,Huang,STW}, to name just a few.

A bank run (or run on the bank) \cite{BTV} takes place when numerous clients withdraw concurrently cash from deposit accounts with a financial institution
because they believe that the financial institution  might be  insolvent. In this case, it is rational to introduce a jump process
to portray   sudden and significant withdrawals. Based on this point of view,   the bank's reserve process can be modelled by a jump diffusion.  Additionally, when the macro-economy suffer from a severe instability,
the phenomenon on  bank runs is contagious, which  leads to the  occurrence of the banking panic \cite{BTV} (i.e., a financial crisis that occurs when many banks suffer runs at the same time). The  observation above demonstrates that the bank's reserves are influenced by a system-wide randomness (e.g., macroeconomic shocks) affecting all agents. The aforementioned insights motivate us to study  conditional McKean-Vlasov jump-diffusions \cite{BWWY,BWWY2}.

To proceed,  we introduce  the underlying probability space we are going to work  on as well as  some notation.
Let $(\OO^1, \mathscr F^1, (\mathscr F_t^1)_{t\ge0}, \P^1)$ and $(\OO^0, \mathscr F^0,(\mathscr F_t^0)_{t\ge0}, \P^0)$ be complete filtered  probability spaces, on which L\'evy processes  $(Z_t)_{t\ge0}$ and   $(Z^0_t)_{t\ge0}$, involved in \eqref{EW0} below,     are  respectively supported. Throughout this paper,  we shall  work on the product  probability space $(\OO, \mathscr F, \mathbb F,\P)$,
where $\OO:=\OO^0\times \OO^1$, $(\mathscr F, \P)$ is the completion of $(\mathscr F^0\otimes \mathscr F^1, \P^0\otimes \P^1)$,  and $\mathbb F$ is the complete  and right-continuous augmentation of $(\mathscr F^0_t\otimes \mathscr F^1_t)_{t\ge0}$. $\scr{P}(\R^d )$ stands for  the family of probability measures on $\R^d $.

In this work,  we focus  on  the following conditional McKean-Vlasov SDE in $\R^d$:
\begin{align}\label{EW0}
\d X_t=b(X_t,\mu_t)\,\d t   +\si\,\d Z_t   +\si_0\,\d Z^0_t,
\end{align}
where $b:\R^d \times\scr{P}(\R^d )\to\R^d $, $\si,\si_0\in\R$,
$(Z_t)_{t\ge0}$ and $(Z^0_t)_{t\ge0}$ are independent $d$-dimensional rotationally  invariant pure jump L\'evy processes, and $\mu_t:=\scr L_{X_t|\scr F_t^0}$. In \eqref{EW0},
$(Z_t)_{t\ge0}$ and $(Z^0_t)_{t\ge0}$ are called the idiosyncratic noise (e.g. bank-specific defaults) and the common noise  (e.g., market-wide shocks), respectively. Throughout the paper, we assume that the respective L\'evy measures $\nu$ and $\nu^0$ associated with $(Z_t)_{t\ge0}$ and $(Z^0_t)_{t\ge0}$ fulfil
 the following integrability conditions:
\begin{equation}\label{vv0}
\int_{\R^d}(|z|\wedge|z|^2)\,\nu(\d z)<\8 \et \int_{\R^d}(|z|\wedge|z|^2)\,\nu^0(\d z)<\8.
\end{equation}

So far, concerning conditional McKean-Vlasov jump-diffusions,
great progress has been made on e.g. well-posedness,  deep learning, optimal stopping, optimal/impulse control, conditional PoC,  stochastic maximum principles;
see e.g. \cite{AR,AO,AO2,BLW,BWWY,BWWY2,HR,HR2} for related details.
However, the exploration on long-time behavior of conditional McKean-Vlasov jump-diffusions
is rare. As shown in Proposition \ref{pro-} below, the conditional distribution flow $(\mu_t)_{t\ge0}$ associated with \eqref{EW0}
solves a stochastic Fokker-Planck equation (SFPE for short), which indeed is a stochastic partial integral differential equation driven by a Poisson random measure.
In the present work,  we   move forward and fill particularly  a gap in investigating the exponential ergodicity of the infinite-dimensional  measure-valued process $(\mu_t)_{t\ge0}$ in lieu of the finite-dimensional process $(X_t)_{t\ge0}$ determined by \eqref{EW0}.

Due to the technical reason, which will be dwelled on in Remark \ref{remark} below, we are confined to
the conditional McKean-Vlasov jump diffusion \eqref{EW0} in $\R$ to state the reasonable hypotheses and the subsequent main result.

\subsection{Main result}
We   assume that
\begin{enumerate}\it
\item[{\rm(${\bf H}_1$)}]
$b(\cdot,\delta_{0})$ is continuous
on $\R$, and
there exist constants $\lambda_1,\lambda_2,\lambda_3>0$ and $\ell_0\ge1$ such that for all $x,y\in\R $ and $\mu,\bar{\mu}\in\mathscr P_1(\R)$,
\begin{equation}\label{H1}
(x-y) (b(x,\mu)-b(y,\mu)) \le (\lambda_1+\lambda_2)| x-y |^2\I_{\{| x-y |\le\ell_0\}} -\lambda_2| x-y |^2,
\end{equation}
and
\begin{align}\label{H2}
|b(x,\mu)-b(x,\bar{\mu})|\le \lambda_3\mathbb W_1(\mu,\bar{\mu}).
\end{align}

\item[{\rm(${\bf H}_2$)}]
 for any conditionally independent and identically distributed   $(X^i_t)_{1\le i\le n}$ under the filtration $\mathscr F_t^{0}$,
 there exists a function $\varphi:[0,\8)\to[0,\8)$ satisfying  $\lim_{r\to\8}\varphi(r)=0$ such that for any $n\ge1,$
\begin{align}\label{H3}
 \max_{1\le i\le n}\sup_{t\ge0}\E|b(X^i_t,\mu_t^i)-b(X^i_t,\tilde \mu^{n,i}_t)| \le \varphi(n),
\end{align}
 where $\mu_t^i:=\mathscr L_{X_t^i|\mathscr F_t^{0}}$ and $\tilde \mu^{n,-i}_t
 :=\frac{1}{n-1}\sum_{j=1,j\neq i}^n\delta_{X^j_t}$.

\item[{\rm(${\bf H}_3$)}]
there exists a function $F_{\si,\si_0}:[0,\infty)\rightarrow[0,\infty)$ such that
\begin{align}\label{EEE}
  F_{\si,\si_0}(r)
   \le \si^2\int_{\{|z|<\frac{1}{2|\si|}r\}}|z|^2\nu(\d z)+\si^2_0\int_{\{|z|<\frac{1}{2|\si_0
   }r\}}|z|^2\nu^0(\d z),\quad r\in[0,2\ell_0],
 \end{align}
and
 $[0,\8)\ni r\mapsto g_*(r):=\lambda_1\int_0^r\frac{s}{F_{\si,\si_0}(s)}\,\d s
 <\infty$
  satisfies that $g_*''(r)\le0$, $g_*^{(3)}(r)\ge0$ and $g_*^{(4)}(r)\le0$ for all $r\in(0,2\ell_0]$.
\end{enumerate}

\ \

Below, we make some comments on Assumptions (${\bf H}_1$),
(${\bf H}_2$) and (${\bf H}_3$).
\begin{remark}
 \eqref{H1} and  \eqref{H2}  indicate respectively that  $b$ is spatially dissipative in long distance, and
  uniformly (in the state variable) continuous in the measure variable under the $L^1$-Wasserstein distance. Under (${\bf H}_1$), via  the fixed point theorem, the SDE \eqref{EW0} admits a unique strong solution even for the multidimensional setting (i.e., $d\ge2$);  see, for instance,  \cite[Theorem 2.1]{BLW} under the weak  monotonicity and  the weak coercivity.
(${\bf H}_1$), besides  (${\bf H}_2$),  enables us to derive  the asymptotic  PoC
in an infinite-time horizon
(see Proposition \ref{pro3} below for more details). Additionally, some  sufficiencies  are furnished in
 \cite[Lemma 4.1]{BW} for the validity of (${\bf H}_2$).
There are a number of examples on $F_{\si,\si_0}$ satisfying (${\bf H}_3$); see, for instance, Example \ref{exa} below for a concrete one.
\end{remark}

Before the presentation of the main result, it
further
necessitates  to introduce some  notation. For a Polish space $(E,\|\cdot\|_E)$,   $\mathscr P(E)$ means  the set of probability measures on $E$ and write $\mathscr P_1(E)$ as
$$\mathscr P_1(E)=\big\{\mu\in\mathscr P(E): \mu(\|\cdot\|_E)<\8\big\}.$$
Set
\begin{align*}
L_1(\mathscr P(\R^d)): =\bigg\{\mu\in\mathscr P(\mathscr P(\R^d)):\int_{\mathscr P(\R^d)}\nu(|\cdot| )\,\mu(\d\nu)<\8\bigg\}
\end{align*}
and define the $L^1$-Wasserstein distance on $L_1(\mathscr P(\R^d))$ as below:
\begin{align*}
\mathcal W_1(\mu_1,\mu_2) =\inf_{\pi\in\mathscr C(\mu_1,\mu_2)}\int_{\mathscr P(\R^d)\times\mathscr P(\R^d)}\mathbb W_1(\tilde {\mu}_1,\tilde{\mu}_2)\,\pi(\d\tilde{\mu}_1,\d\tilde{\mu}_2),\quad \mu_1,\mu_2\in L_1(\mathscr P(\R^d)),
\end{align*}
where $\mathscr C(\mu_1,\mu_2)$ means  the family  of couplings of $\mu_1,\mu_2$, and
$\mathbb W_1$ embodies the
$L^1$-Wasserstein distance, which is defined as follows:
$$
\W_1(\mu_1,\mu_2):=\inf_{\pi\in\mathscr{C}(\mu_1,\mu_2)}\left(\int_{\R^d\times\R^d}|x-y|\,\pi(\d x,\d y)\right),\quad \mu_1,\mu_2\in\mathscr P_1(\R^d).
$$

The main result in the present work is stated as below, which demonstrates that the measure-valued process $(\mu_t)_{t\ge0}$ is weakly contractive under the $L^1$-Wasserstein distance $\mathcal W_1$.

\begin{theorem}\label{thm1}
Assume that $({\bf H }_1)$,
$({\bf H}_2)$ and $({\bf H }_3)$ hold and suppose $\si,\si_0\neq0$.
Then,
there exist    constants $C,\lambda_0^*,\lambda_3^*>0$  satisfying that  for  any  $t\ge0$   and $\lambda_3\in[0,\lambda_3^*]$,
\begin{equation}\label{EE1}
\mathcal W_1(\mathscr L_{\mu_t},\mathscr L_{\bar{\mu}_t})\le C\e^{-\lambda_0^* t}
\mathcal W_1(\mathscr L_{\mu_0},\mathscr L_{\bar{\mu}_0}),
\end{equation}
where  $\mu_t: =\mathscr L_{X_t|\mathscr F_t^{0}}$ and   $\bar{\mu}_t: =\mathscr L_{X_t|\mathscr F_t^{0}}$ stands for  the regular
conditional distributions of $X_t$, determined by the conditional McKean-Vlasov SDE   \eqref{EW0} in $\R$,
 with   initial distributions
$\mathscr L_{\mu_0}$ and $\mathscr L_{\bar{\mu}_0}$, respectively;
   $\lambda_3>0$ is the Lipschitz constant, given in \eqref{H2},  of $b(x,\mu)$ in  the measure variable.
\end{theorem}

\begin{remark}
 By invoking the weak contraction \eqref{EE1} and applying the Banach fixed point theorem, the measure-valued process $(\mu_t)_{t\ge0}$ associated with the conditional McKean-Vlasov jump diffusion  \eqref{EW0} in $\R$
 has a unique
invariant probability measure (which is also called a stationary distribution)
 provided that the mean-field interaction is not too strong (i.e., $\lambda_3>0$ in \eqref{H2} is small enough).

For classical McKean-Vlasov SDEs without common noise, the study of ergodicity is explored by means of the corresponding decoupled SDEs  as shown in \cite{Eberle,LMW,Wanga}. Nonetheless, as far as conditional McKean-Vlasov SDEs are concerned,  the routine taken in \cite{Eberle,LMW,Wanga}  is no longer workable. In turn, inspired by \cite{BW,Maillet}, we appeal to the associated stochastic interacting particle system to tackle the ergodicity of the measure-valued process $(\mu_t)_{t\ge0}$ associated with \eqref{EW0} in $\R$.
\end{remark}

In comparison with the existing literature \cite{BW,Maillet}, the innovation of the present paper lies in the following two aspects.
\begin{remark}
 (1) {\it Framework.} In contrast to \cite{BW,Maillet}, the driven noises involved in this paper  are totally different. In detail,
in \cite{BW,Maillet} the  idiosyncratic  noise and the
common noise are independent Brownian motions. In this context, the conditional distribution flow satisfies an  SFPE, which in fact  is a stochastic partial differential equation driven by Brownian motion. Concerning the conditional McKean-Vlasov jump diffusion \eqref{EW0} in $\R$, the underlying idiosyncratic  noise and the
common noise are jump processes. Correspondingly, the conditional distribution flow also fulfils an SFPE, which nevertheless is a stochastic partial {\it integral} equation driven by a {\it Poisson random measure}.

(2) {\it Coupling construction.} Regarding the work \cite{Maillet}, the reflection coupling and the synchronous coupling  were  applied respectively  to the common noise and
the idiosyncratic noise.  As for \cite{BW}, the reflection coupling was employed to not only the common noise but  also the
 idiosyncratic noise whereas, with regarding  to the multiplicative noise, the  synchronous coupling was adopted. With the aid of a well-chosen  threshold, the whole jump size is divided into two parts, where one part is the so-called small-size part and the other one is the big-size part.
When the associated jump size is located in the small-size zone, the asymptotic coupling by reflection  is explored. On the contrary,
 the  synchronous coupling  is taken into account.
\end{remark}

 In the past few years, since the seminal work  \cite{LW}, the ergodicity of (McKean-Vlasov) SDEs driven by non-symmetric L\'{e}vy processes
 has been investigated considerably; see, for instance, \cite{LMW} and references within. Whereas, in the present work,
 the establishment of our main result (i.e., Theorem \ref{thm1})  is assumed that
 both the jump  idiosyncratic  noise and the jump common noise possess the  rotationally invariant property, which plays an important role in constructing the asymptotic coupling by reflection; see in particular the proof of Proposition \ref{pro2} for related details. As a continuation of the present work, it is quite natural to seek an extension to the case that   the    idiosyncratic  noise and the   common noise are
  non-symmetric L\'evy noises. Concerning such setting,  the construction of the underlying  coupling might  be fundamentally different and more intricate. This is left to explore  in our future work.

\ \

The rest of this paper is arranged as follows. In Section \ref{sec2}, we (i) show that the conditional distribution flow solves an SFPE driven by a Poisson random measure,  (ii) reveal  that the conditional distribution flow associated with the stochastic non-interacting particle system
keep untouch with respect to the particle index,  (iii) establish
 the conditional PoC in a finite-time horizon, as well as (iv) construct an asymptotic  coupling process
for the associated stochastic non-interacting particle system and the stochastic  interacting particle system.  Section \ref{sec3} is devoted to the proof of  Theorem \ref{thm1}, which is treated on account of
 the uniform-in-time conditional PoC  for the conditional McKean-Vlasov jump diffusion \eqref{EW0} in $\R$.

\section{Preliminaries}\label{sec2}
In this section, for the conditional McKean-Vlasov jump diffusion \eqref{EW0} in $\R^d$ (rather than $\R$),
we set up  a   series of preparatory work, which lays the foundation  of the proof for Theorem \ref{thm1}.
Roughly speaking, in Subsection \ref{subsec1}, we show that the measure-valued process $(\mu_t)_{t\ge0}$ solves a stochastic partial integral equation driven by a Poisson random measure. In addition, we demonstrate that the corresponding $((\mu_t^i)_{t\ge0})_{ 1\le i\le n}$
coincide almost surely with $(\mu_t)_{t\ge0}$ when, in the stochastic non-interacting particle system,  the  idiosyncratic noise $(Z_t)_{t\ge0}$  is replaced by   i.i.d. copies $((Z_t^i)_{t\ge0})_{ 1\le i\le n} $ whereas the common noise $(Z_t^0)_{t\ge0}$ is kept untouch.  Our goal in the other subsections  is twofold, where the former  one is to investigate the conditional PoC in finite time, and the latter one  is to construct the  so-called  asymptotic coupling by reflection.

Throughout this section,  we always suppose that
\begin{enumerate} \it
\item[{\rm(${\bf A}_1$)}] $b(\cdot,\delta_{\bf0}):\R^d  \to\R^d $  is continuous on $\R^d $, and  there exist     constants  $L_1,L_2 >0$ such that for all $x,y\in\R^d $ and $\mu,\bar{\mu}\in\mathscr P_1(\R^d)$,
\begin{equation}\label{E1}
\<x-y, b(x,\mu)-b(y,\mu)\>\le L_1|x-y|^2,
\end{equation}
and
\begin{align}\label{E2}
|b(x,\mu)-b(x,\bar{\mu})|\le L_2\mathbb W_1(\mu,\bar{\mu}).
\end{align}
\end{enumerate}

It is easy to see that  Assumption (${\bf A}_1$) implies that  for all $x,y\in\R^d $ and $\mu,\bar{\mu}\in\mathscr P_1(\R^d)$,
\begin{equation}\label{E3}
\<x-y, b(x,\mu)-b(y,\bar{\mu})\>\le (L_1\vee L_2)\big(|x-y|+\mathbb W_1(\mu,\bar{\mu})\big)|x-y|.
\end{equation}
Then,
the SDE \eqref{EW0}   has a unique
strong solution; see e.g. \cite[Theorem 4.1]{BLW} for related details.

In \eqref{EW0}, if $(Z_t)_{t\ge0}$ is replaced by i.i.d. copies $((Z_t^i)_{t\ge0})_{1\le i\le n}$,  supported on  $(\OO^1, \mathscr F^1, (\mathscr F^1_t)_{t\ge0},\P^1)$,   the following non-interacting particle system:
\begin{equation}\label{EW1}
\d X_t^i= b(X_t^i,\mu_t^i) \,\d t+   \si\,\d Z_t^i  +\si_0\,\d Z^{0}_t, \quad 1\le i\le n
\end{equation}
is available, in which  $ \mu_t^i:=\mathscr L_{X_t^i|\mathscr F_t^{0}} $. Furthermore, if we replace $\mu_t^i$ in \eqref{EW1} with the associated empirical measure $\hat\mu_t^{n}:=\frac{1}{n}\sum_{j=1}^n\delta_{X_t^{j,n}}$,  the stochastic interacting particle system
\begin{equation}\label{EW2}
\d X_t^{i,n}=b(X_t^{i,n},\hat\mu_t^n) \,\d t+  \si\,\d Z_t^i  +\si_0\,\d Z^{0}_t,\quad 1\le i\le n
\end{equation}
is attainable.  \eqref{EW2} is indeed a classical  $(\R^{d})^n$-valued SDE, which is strongly well-posed (see e.g.  \cite[Theorem 1.1]{BLW})
under Assumption (${\bf A}_1$)
by taking advantage of the fact that  the lifted drift satisfies the so-called weak monotonicity and the weak coercivity.
 Additionally, in the subsequent analysis, it is assumed  that the initial value
$(X_0^{i}, X_0^{i,n})_{1\le i\le n}$
are i.i.d.\ $\mathscr F_0$-measurable random variables.

\subsection{Stochastic Fokker-Planck equation and invariance of $(\mu_\cdot^i)_{1\le i\le n}$}\label{subsec1}
In this subsection, in the first place,
we  aim at showing that the conditional distribution flow $(\mu_t)_{t\ge0}$ solves
an   SFPE, which indeed is a stochastic partial integral differential equation driven by a Poisson random measure.
To start, by means of  the L\'{e}vy-It\^o decomposition, $(Z_t)_{t\ge0}$ and $(Z_t^0)_{t\ge0}$ can be written respectively  as below: for any $t>0,$
\begin{align*}
Z_t=\int_0^t\int_{\{|z|\le1\}}z\tilde N(\d s,\d z)+\int_0^t\int_{\{|z|>1\}}z  N(\d s,\d z)
\end{align*}
and
\begin{align*}
Z_t^0=\int_0^t\int_{\{|z|\le1\}}z\tilde N^0(\d s,\d z)+\int_0^t\int_{\{|z|>1\}}z  N^0(\d s,\d z),
\end{align*}
where $N(\d s,\d z)$ and $N^0(\d s,\d z)$ are  Poission random measures,  supported  on $(\OO^1, \mathscr F^1, (\mathscr F_t^1)_{t\ge0}, \P^1)$ and $(\OO^0, \mathscr F^0,(\mathscr F_t^0)_{t\ge0}, \P^0)$, with  L\'{e}vy measures $\nu(\d z)$ and $\nu^0(\d z)$, respectively.

\begin{proposition}\label{pro-}
The conditional distribution flow $(\mu_t)_{t\ge0}$ solves   the following SFPE:
\begin{equation}\label{T3}
\begin{split}
\d \mu_t&=-{\rm div}(b(\cdot,\mu_t)\mu_t)\,\d t+ \int_{\R^d} \big( \delta_{\sigma z}\ast\mu_t -\mu_t +\si{\rm div}(z\mu_t)\I_{\{|z|\le1\}}\big)\,\nu(\d z)\,\d t\\
&\quad+\int_{\R^d} \big( \delta_{\sigma_0 z}\ast\mu_t -\mu_t +\si_0{\rm div}(z\mu_t)\I_{\{|z|\le1\}}\big)\,\nu^0(\d z)\,\d t\\
&\quad+\int_{\R^d}  (\delta_{\sigma z}\ast\mu_t)(\varphi)-\mu_t(\varphi)\big)  \tilde N^0(\d t,\d z),
\end{split}
\end{equation}
where, for $x\in\R^d$, the probability measure $\delta_x\ast \mu_t$ stands for the convolution between
 $\delta_x$ and $\mu_t.$ The solution to \eqref{T3} is understood in the sense of distribution, that is, for any $\varphi\in C_c^2(\R^d)$,
\begin{equation}\label{T4}
\begin{split}
\d \mu_t(\varphi)&=\mu_t(\<\nn \varphi(\cdot),b(\cdot,\mu_t)\>)\,\d t\\
&\quad+   \int_{\R^d} \big((\delta_{\sigma z}\ast\mu_t)(\varphi)-\mu_t(\varphi)-\si\mu_t(\<\nn\varphi(\cdot),z\>)\I_{\{|z|\le1\}}\big)\,\nu(\d z)\,\d t \\
&\quad+ \int_{\R^d}\big((\delta_{\sigma_0 z}\ast\mu_t)(\varphi)-\mu_t(\varphi)-\si_0\mu_t(\<\nn\varphi(\cdot),z\>)\I_{\{|z|\le1\}}\big) \,\nu^0(\d z)\,\d t\\
&\quad+\int_{\R^d}  (\delta_{\sigma z}\ast\mu_t)(\varphi)-\mu_t(\varphi)\big)  \tilde N^0(\d t,\d z).
\end{split}
\end{equation}

\end{proposition}

\begin{proof}
For any $\varphi\in C_c^2(\R^d)$, by applying It\^o's formula, we deduce from \eqref{EW0} that   for any $t\ge0,$
\begin{align*}
\varphi(X_t)&=\varphi(X_0)+\int_0^t\<\nn\varphi(X_s),b(X_s,\mu_s)\>\,\d s\\
&\quad+\int_0^t\int_{\R^d}\big(\varphi(X_s+\si z)-\varphi(X_s)-\si\<\nn\varphi
(X_s),z\>\I_{\{|z|\le1\}}\big)\,\nu(\d z)\,\d s\\
&\quad+\int_0^t\int_{\R^d}\big(\varphi(X_s+\si_0 z)-\varphi(X_s)-\si_0\<\nn\varphi
(X_s),z\>\I_{\{|z|\le1\}}\big)\,\nu^0(\d z)\,\d s\\
&\quad+\int_0^t\int_{\R^d}\big(\varphi(X_s+\si z)-\varphi(X_s)\big)\tilde N(\d s,\d z)\\
&\quad+\int_0^t\int_{\R^d}\big(\varphi(X_s+\si_0 z)-\varphi(X_s)\big)\tilde N^0(\d s,\d z)\\
&= :\varphi(X_0)+\sum_{i=1}^5I_i(t).
\end{align*}
Subsequently, for given $t\ge0,$ taking conditional expectations with respect to $\mathscr F_t^{0}$ yields that
\begin{align*}
\E\big(\varphi(X_t)\big|\mathscr F_t^{0}\big)= \E\big(\varphi(X_0)\big|\mathscr F_t^{0}\big)+\sum_{i=1}^5\E\big(I_i(t)\big|\mathscr F_t^{0}\big).
\end{align*}
Set $\mathscr F_t^X:=\sigma(X_s: s\le t)$, i.e.,
the sigma algebra generated by $(X_s)_{s\ge0}$ up to time $t.$  For any $0\le s\le t, $
since $\mathscr F_s ^X$ is conditionally independent of $\scr F_t^{ 0}$ conditioned on $\scr F_s^{ 0}$, we have
\begin{align*}
\mu_s =\scr L_{X_s|\scr F_t^{ 0}},\quad \mbox{ a.s.}, \quad 0\le s\le t.
\end{align*}
Whence, we find that
\begin{align}\label{T1}
&\E\big(\varphi(X_0)\big|\mathscr F_t^{ 0}\big)+\sum_{i=1}^3\E\big(I_i(t)\big|\mathscr F_t^{ 0}\big)\nonumber\\
&=\mu_0(\varphi)+\int_0^t\mu_s(\<\nn \varphi(\cdot),b(\cdot,\mu_s)\>)\,\d s\nonumber\\
&\quad+  \int_0^t\int_{\R^d}\mu_s\big(\varphi(\cdot+\si z)-\varphi(\cdot)-\si\<\nn\varphi(\cdot),z\>\I_{\{|z|\le1\}}\big)\,\nu(\d z)\,\d s\nonumber\\
&\quad+\int_0^t\int_{\R^d}\mu_s\big(\varphi(\cdot+\si_0 z)-\varphi(\cdot)-\si_0\<\nn\varphi(\cdot),z\>\I_{\{|z|\le1\}}\big)\,\nu^0(\d z)\,\d s\nonumber\\
&=\mu_0(\varphi)+\int_0^t\mu_s(\<\nn \varphi(\cdot),b(\cdot,\mu_s)\>)\,\d s\\
&\quad+  \int_0^t\int_{\R^d} \big((\delta_{\sigma z}\ast\mu_s)(\varphi)-\mu_s(\varphi)-\si\mu_s(\<\nn\varphi(\cdot),z\>)\I_{\{|z|\le1\}}\big)\,\nu(\d z)\,\d s\nonumber\\
&\quad+\int_0^t\int_{\R^d}\big((\delta_{\sigma_0 z}\ast\mu_s)(\varphi)-\mu_s(\varphi)-\si_0\mu_s(\<\nn\varphi(\cdot),z\>)\I_{\{|z|\le1\}}\big) \,\nu^0(\d z)\,\d s. \nonumber
\end{align}
Via an approximation trick, besides the independence between $(Z_t)_{t\ge0}$ and $(Z_t^0)_{t\ge0}$,
it is easy to see that
\begin{align}\label{T2}
\E\big(I_4(t)\big|\mathscr F_t^{ 0}\big)=0.
\end{align}
Next, by repeating exactly  the proof of \cite[Lemma B.1]{LSZ}, we derive that  for any $t\ge0,$
\begin{align*}
 \E\big(I_5(t)\big|\mathscr F_t^{ 0}\big) &=\int_0^t\int_{\R^d}\mu_s\big( \varphi(\cdot+\si z)-\varphi(\cdot)\big)  \tilde N^0(\d s,\d z)\\
&=\int_0^t\int_{\R^d}  (\delta_{\sigma z}\ast\mu_s)(\varphi)-\mu_s(\varphi)\big)  \tilde N^0(\d s,\d z).
\end{align*}
 This, combining \eqref{T1} with \eqref{T2}, yields \eqref{T4} so that \eqref{T3} follows directly.
\end{proof}

\begin{remark}
When the common noise is a standard Brownian motion and the idiosyncratic noise is a compensated Poisson process,
the associated SFPE has been established in \cite[Theorem 2.2]{AO} and
 \cite[Theorem 3.3]{AO2} via the  Fourier transformation. Nonetheless, we herein finish the proof of Proposition \ref{pro-} by the aid of an alternative approach which is inspired by that of \cite[Proposition 1.2]{LSZ}, where both the common noise  and  the idiosyncratic noise are Brownian motions.
\end{remark}

The following proposition reveals the fact that $(\mu_t^i)_{1\le i\le n}$ are unchanging almost  surely provided that the associated jump idiosyncratic  noises are independent and identically distributed, and that the   jump common   noise remains unchanged.

\begin{proposition}\label{pro4}
 Under $({\bf A}_1)$, for any given $T>0 $ and all $i=1,\cdots, n,$
\begin{align*}
\P^0\big(\mu_t=\mu_t^i\mbox{ for all } t\in[0,T]\big)=1,
\end{align*}
where $(\mu_t)_{t\ge0}$ and $(\mu_t^i)_{t\ge0}$ are conditional distribution flow associated with \eqref{EW0} and \eqref{EW1}, respectively.
\end{proposition}

\begin{proof}
Since the proof  is similar to   that of \cite[Proposition 2.11]{CD2},  we herein give merely  a sketch   to make the content self-contained.

For fixed $T>0$ and a Polish space $\mathbb U$, let $D([0,T];\mathbb U)$ be the collection of functions $f:[0,T]\to\mathbb U$, which are right-continuous with left limits. For $\xi\in D([0,T];\mathbb U)$, we write $\xi_{[0,T]}$ as the path of $\xi$ up to $T.$ In the following analysis, we fix $1\le i\le n $ and the terminal $T$.
Under $({\bf A}_1)$,  the SDE \eqref{EW1} is strongly well-posed so that  there exists a measurable map: $$\Phi:\R^d\times   D([0,T];\R^d)\times D([0,T];\mathscr P_1(\R^d))\times D([0,T];\R^d)\to D([0,T];\R^d)$$
such that
\begin{align*}
\P\big(X^i_{[0,T]}=\Phi(X_0, Z^0_{[0,T]},\mu^i_{[0,T]},Z^i_{[0,T]})\big)=1.
\end{align*}
For $\mu^i_{[0,T]}$ given  previously, consider the following decoupled SDE:
\begin{align}\label{ET6}
\d U_t^i=b(U_t^i,\mu^i_t)\,\d t+\si\,\d Z_t  +\si_0\,\d Z^{0}_t,\quad t\in[0,T]; \quad U_0^i=X_0.
\end{align}
Once more, via the strong well-posedness of \eqref{ET6}, we have
\begin{align*}
\P\big(U^i_{[0,T]}=\Phi(X_0, Z^0_{[0,T]},\mu^i_{[0,T]},Z_{[0,T]})\big)=1.
\end{align*}
Due to the fact that $(Z_t)_{t\ge0}$ and $(Z_t^i)_{t\ge0}$, supported on $(\Omega^1,\mathscr F^1,\mathbb P^1)$, are identically distributed,
we find that for $\mathbb P^0$-a.s. $\omega^0\in\Omega^0$,
\begin{align}\label{TE}
\mathscr L_{U_t^i(\omega^0,\cdot)}=(\mathscr L_{U^i_t|\mathscr F_t^0})(\omega^0)=\mu_t^i(\omega^0),\quad t\in[0,T].
\end{align}
Whence, we arrive at
\begin{align*}
\P\big(U^i_{[0,T]}=X_{[0,T]}\big)=1.
\end{align*}
This, along with \eqref{TE},  further yields that
\begin{align*}
\mu_t(\omega^0)=\mathscr L_{X_t(\omega^0,\cdot)}=\mathscr L_{U_t^i(\omega^0,\cdot)}=\mu_t^i(\omega^0),\quad t\in[0,T].
\end{align*}
Thus, the proof is complete.
\end{proof}

\subsection{Conditional PoC in  finite  time}

In the past few decades,   the issue on the convergence rate of the (conditional) PoC in a finite horizon
concerning (conditional) McKean-Vlasov SDEs driven by L\'{e}vy processes has
been studied extensively.   In particular, we allude to e.g. \cite[Proposition 3.1]{NBKG} and \cite[Proposition 3.2]{KKLN}, in which  the L\'{e}vy measure involved  enjoys a higher-order moment. In case the conditional McKean-Vlasov SDEs driven by  the   jump L\'{e}vy process with the heavy-tailed property,
we refer to  \cite[Theorem 2]{Cava} and \cite[Theorem 1.3]{BLW} focusing on  the conditional PoC,
where the drift terms   under consideration   fulfil   the  Lipschitz continuity and  the   weak monotonicity, respectively. No matter what \cite{BLW,Cava}  or  \cite{KKLN,NBKG}, the higher-order moment of the initial distribution is necessitated
to obtain  the desired convergence rate of   the conditional PoC. Nevertheless, in the present work the
qualitative
 convergence (instead of the quantitative convergence rate) of the conditional PoC is sufficient to realize our desired goal. In contrast to \cite{BLW,Cava,KKLN,NBKG}, the convergence of the conditional PoC can be reached  under weaker assumptions as shown in the following proposition.

\begin{proposition}\label{pro1}
Let $((X_t^i)_{t\ge0})_{1\le i\le n}$ and $((X_t^{i,n})_{t\ge0})_{1\le i\le n}$ with $X_0^i=X_0^{i,n}, 1\le i\le n, $ be solutions to  \eqref{EW1} and \eqref{EW2}, respectively.
Under $({\bf A}_1)$ and   $\E|X_0^1|<\8$,
\begin{itemize}
\item[{\rm(i)}] for each given $t\ge0$ and any  $1\le i\le n$,
\begin{align}\label{E4}
\lim_{n\to\8}\E\mathbb W_1(\mu_t^i,\tilde\mu_t^n)=0 \quad \mbox{ with } \quad \tilde\mu_t^n:=\frac{1}{n}\sum_{j=1}^n\delta_{X_t^j};
\end{align}
\item[{\rm(ii)}]  for each given $t\ge0$ and any  $1\le i\le n$,
\begin{align}\label{E5}
\lim_{n\to\8}\E|X_t^i-X_t^{i,n}|=0.
\end{align}

\end{itemize}

\end{proposition}

\begin{proof}  To achieve  \eqref{E4} and \eqref{E5}, as a starting point,
we claim  that there exists a constant $c_0>0$ such that for any $t\ge0$ and $1\le i\le n$,
\begin{align}\label{EX}
\E  |X_t^i|\le c_0\big(1+t+\E|X_0^i|\big)\e^{c_0t}.
\end{align}
To this end, we define the Lyapunov function $V(x)=(1+|x|^2)^{\frac{1}{2}}, x\in\R^d$.
By applying It\^o's formula, it is easy to see that
\begin{align*}
 \d V(X_t^i)&=\<\nn V(X_t^{i}),b(X_t^i,\mu_t^i)\>\,\d t
 +\int_{\{|z|\le1\}}\big(V(X_t^i+\si z)-V(X_t^i)-\si\<\nn V(X_t^{i}), z\> \big)\,\nu(\d z)\d t\\
 &\quad+\int_{\{|z|\le1\}}\big(V(X_t^i+\si_0 z)-V(X_t^i)-\si_0\<\nn V(X_t^{i}), z\> \big)\,\nu^0(\d z)\d t \\
 &\quad+\int_{\{|z|>1\}}\big(V(X_t^i+\si z)-V(X_t^i) \big)\,\nu(\d z)\d t\\
 &\quad+\int_{\{|z|>1\}}\big(V(X_t^i+\si_0 z)-V(X_t^i) \big)\,\nu^0(\d z)\d t  +\d M_t^i\\
 &=:\<\nn V(X_t^{i}),b(X_t^i,\mu_t^i)\>\,\d t+(I_t^{1,i} +I_t^{2,i}+I_t^{3,i} +I_t^{4,i} )\,\d t+ \d M_t^i,
\end{align*}
 where $(M_t^i)_{t\ge0}$ is a martingale. By invoking \eqref{E1} and \eqref{E2}, we obviously have for all $x\in\R^d$ and $\mu\in\mathscr P_1(\R^d)$,
 \begin{align}\label{ET}
(1+|x|^2)^{-\frac{1}{2}} \<x,b(x,\mu)\>
 &\le L_1|x| +L_2\mu(|\cdot|)+|b({\bf0},\delta_{\bf0})|.
 \end{align}
 Note that
 \begin{align*}
  \nn V(x)=(1+|x|^2)^{-\frac{1}{2}}x\quad \mbox{ and } \quad \nn^2V(x)=(1+|x|^2)^{-\frac{1}{2}}I_d-(1+|x|^2)^{-\frac{3}{2}}xx^\top,\quad x\in\R^d,
  \end{align*}
 where $x^\top$ denotes the transpose of $x\in\R^d.$
Then, the Taylor expansion enables us to derive that
  \begin{align}\label{ET1}
  I_t^{3,i} +I_t^{4,i}\le |\si| \int_{\{|z|>1\}}|z|\,\nu(\d z)+|\si_0|\int_{\{|z|>1\}}|z|\,\nu^0(\d z),
  \end{align}
 and
 \begin{align}\label{ET2}
    I_t^{1,i} +I_t^{2,i}\le\frac{1}{2} \si^2  \int_{\{|z|\le1\}}|z|^2\,\nu(\d z)+ \frac{1}{2}
    \si_0^2 \int_{\{|z|\le1\}}|z|^2\,\nu^0(\d z).
 \end{align}
 Subsequently, by combining \eqref{ET} with \eqref{ET1} and \eqref{ET2} and making use of
  the fact that
  \begin{align}\label{ET9}
 \E\mu_t^i(|\cdot|)=
  \E^0\mu_t^i(|\cdot|)=
 \E^0\big(\E^1\big(|X_t^i|\big|\mathscr F_t^0\big)\big)=\E|X_t^i|,
 \end{align} there exists a constant $c_1>0$ such that
 \begin{align}\label{ET4}
 \E|X_t^i|\le 1+\E|X_0^i|+2c_1\int_0^t(1+\E|X_s^i|)\d s.
\end{align}
As a consequence,  \eqref{EX} is reachable by applying    Gr\"onwall's inequality.

Notice that
$$\E\mathbb W_1(\mu_t^i,\tilde\mu_t^n)=\E^0\big(\E^1\mathbb W_1(\mu_t^i,\tilde\mu_t^n)\big)\quad \mbox{ and } \quad \E^1\mathbb W_1(\mu_t^i,\tilde\mu_t^n) \le 2 \mu_t^i(|\cdot|).$$
Thus, by applying the dominated convergence theorem,
the assertion \eqref{E4} is available provided that $\E^0\mu_t^i(|\cdot|)<\8$ and
\begin{align}\label{ET3}
\P^0\Big(\lim_{n\to\8}\E^1\mathbb W_1(\mu_t^i,\tilde\mu_t^n)=0\Big)=1.
\end{align}
In fact, $\E^0\mu_t^i(|\cdot|)<\8$ is guaranteed by taking advantage of \eqref{EX}
and \eqref{ET9}.
Next,
since   $\tilde\mu_t^n$
converges weakly to $\mu_t^i$,  $\P^0$-almost surely, and
\begin{equation*}
\P^1\Big(\lim_{n\to\8}\tilde\mu_t^n(|\cdot|)=\mu_t^i(|\cdot|)\Big)=1,
\end{equation*}
we deduce from    \cite[Theorem 5.5]{CD1} that
\begin{align*}
\P^1\Big(\lim_{n\to\8}\mathbb W_1(\mu_t^i,\tilde\mu_t^n)=0\Big)=1, \quad \P^0\mbox{-almost surely}.
\end{align*}
Subsequently, \eqref{ET3} is available by using  the dominated convergence theorem and noting that
\begin{align*}
 \mathbb W_1(\mu_t^i,\tilde\mu_t^n)\le  \mu_t^i(|\cdot|)+\tilde\mu_t^n(|\cdot|)
\end{align*}
as well as  the fact that $X_t^i$ and $X_t^j$ are identically distributed given the filtration $\mathscr F_t^0$.

For notational simplicity, we set $Q_t^{i,n}:=X_t^i-X_t^{i,n}$.
It is easy to see that
$$\d Q_t^{i,n}=(b(X_t^i,\mu_t^i)-b(X_t^{i,n},\hat{\mu}_t^n))\,\d t.$$
By the chain rule, it follows from \eqref{E3} and $X_0^i=X_0^{i,n}$
that for any $\vv>0,$
\begin{align*}
(\vv+|Q_t^{i,n}|^2)^{\frac{1}{2}}&\le \ss\vv +(L_1\vee L_2)\int_0^t\big(|Q_s^{i,n}|+\W_1(\mu_s^i,\hat{\mu}_s^n)\big)\,\d s\\
&\le \ss\vv+(L_1\vee L_2)\int_0^t\Big(|Q_s^{i,n}|+\mathbb W_1(\mu_s^i,\tilde\mu_s^n)+\frac{1}{n}\sum_{j=1}^n|Q_s^{j,n}|\Big)\,\d s.
\end{align*}
This, together with the fact that
 $(X_t^i,X_t^{i,n})_{1\le i\le n}$ are identically distributed
by recalling that $(X_0^{i}, X_0^{i,n})_{1\le i\le n}$
  are i.i.d.
   $\mathscr F_0$-measurable
   random variables,
gives that
\begin{align*}
\E(\vv+|Q_t^{i,n}|^2)^{\frac{1}{2}}
&\le \ss\vv+(L_1\vee L_2)\int_0^t\big(2\E|Q_s^{i,n}|+\E\mathbb W_1(\mu_s^i,\tilde\mu_s^n) \big)\,\d s.
\end{align*}
At length,   \eqref{E5}   holds true  from  Gr\"onwall's inequality followed by leveraging \eqref{E4} and sending $\vv\to0$.
\end{proof}

\subsection{Asymptotic coupling by reflection}\label{section2.2}
In the beginning, we introduce some additional notation.
For given $\vv>0,$ define a cut-off function $h_\vv$  as below:
 \begin{equation}\label{E6}
h_\vv(r)=
\begin{cases}
0,\qquad \qquad\qquad\qquad\qquad \qquad  \qquad \qquad r\in[0,\vv],\\
6\Big(\frac{r-\vv}{\vv}\Big)^5-15\Big(\frac{r-\vv}{\vv}\Big)^4+10\Big(\frac{r-\vv}{\vv}\Big)^3,\quad\quad~ r\in(\vv,2\vv),\\
1,\qquad \qquad\qquad\qquad\qquad\quad\quad \qquad \qquad ~  r\ge 2\vv.
\end{cases}
\end{equation}
The unit vector ${\bf n}(x)$ related to  $x\in\R^d$ is defined in the form:
 $${\bf n}(x):=\frac{x}{|x|}\I_{\{x\neq {\bf0}\}}+(1,0,\cdots,0)^\top\I_{\{x= {\bf0}\}}.$$
In this subsection, we postulate that
$\rho:(\R^d)^n\to[0,\8)$ and  $\phi:(\R^d)^n\to\R^d,$
whose explicit expressions will be given explicitly in Section \ref{sec3}. In addition, for $\vv>0,$
 we define the approximate reflection matrix $\Pi_{\vv}$ as follows:
 for any
${\bf x}:=(x^1,\cdots,x^n)\in  (\R^d)^n$,
\begin{equation}\label{Pi}
\Pi_{\vv,d}({\bf x}):=I_d-2h_\vv(\rho({\bf x})){\bf n}(\phi( {\bf x}))\otimes {\bf n}(\phi({\bf x})).
\end{equation}
Specifically,
for the case $d=1$,
  $\Pi_{\vv,1}({\bf x}) =1-2h_\vv(\rho({\bf x})),$  which is independent of the choice of the function $\phi$.

Before we move on to  construct the asymptotic coupling by reflection
associated with  the stochastic non-interacting particle system \eqref{EW1} and the corresponding stochastic interacting particle system \eqref{EW2}, some warm-up work need to done.  Via the L\'evy-It\^o decomposition, for each fixed $i=0,1,\cdots, n,  $ $(Z_t^i)_{t\ge0}$ can be expressed as below:
$$
Z_t^i=\int_0^t\int_{\{|z|>1\}}z\,N^i(\d s,\d z)+\int_0^t\int_{\{|z|\le 1\}}z\,\wt{N}^i(\d s,\d z), \quad t\ge0,
$$
where $N^i(\d s,\d z)$ is the Poisson random measure with the
common
intensity measure $\d s\nu(\d z)$, and $\wt{N}^i(\d s,\d z)$ is the corresponding  compensated Poisson random measure, i.e.,
$$
\wt{N}^i(\d s,\d z)=N^i(\d s,\d z)-\d s\nu(\d z),\quad i=0,1,\cdots,n.
$$
In the sequel, for the sake of simplicity,  we write  $$\ol{N}^i(\d t,\d z) =\I_{(0,1]}(|z|)\,\wt{N}^i(\d s,\d z)+\I_{(1,\8)}(|z|)\,N^i(\d s,\d z),\quad i=0,1,\cdots,n
.$$
Correspondingly, we have
$$Z_t^i=\int_{\R^d}z\,\ol{N}^i(\d t,\d z), \quad i=0,1,\cdots,n.$$

With the previous notation at hand,   we build the following  approximate stochastic  interacting particle system: for   $i=1,\cdots,n $ and $\vv>0,$
\begin{equation}\label{EW3}
\begin{cases}
\d  X_t^i=b(X_t^i,\mu_t^i) \d t+  \si\,\d Z_t^i  +\si_0\,\d Z^{0}_t,\\
\d  X_t^{i,n,\vv}=b(X_t^{i,n,\vv},\hat\mu_t^{n,\vv}) \d t\\
    \qquad\qquad\,+\si\displaystyle\int_{\{|z|\le \frac{1}{2|\si|}|Z_t^{i,n,\vv}|\}}\Pi_{\vv,d}
    ({\bf Z}_t^{n,\vv})z\,\ol{N}^i(\d t,\d z)
     +\si\displaystyle\int_{\{|z|>\frac{1}{2|\si|}|Z_t^{i,n,\vv}|\}}z\,\ol{N}^i(\d t,\d z)\\
    \qquad\qquad\,+\si_0\displaystyle\int_{\{|z|\le\frac{1}{2|\si_0|}|Z_t^{i,n,\vv}|\}}\Pi_{\vv,d}
    ({\bf Z}_t^{n,\vv})z\,\ol{N}^0(\d t,\d z)
     +\si_0\displaystyle\int_{\{|z|>\frac{1}{2|\si_0|}|Z_t^{i,n,\vv}|\}}z\,\ol{N}^0(\d t,\d z),\\
\end{cases}
\end{equation}
where   $X_0^{i,n,\vv} = X_0^{i,n} $,  $(X_0^{i}, X_0^{i,n})_{1\le i\le n}$
  are i.i.d.\ $\mathscr F_0$-measurable
   random variables,
$\hat\mu_t^{n,\vv}:=\frac{1}{n}\sum_{j=1}^n\delta_{  X_t^{j,n,\vv}}$,  $ Z_t^{i,n,\vv}:=X_t^{i}-X_t^{i,n,\vv}$, ${\bf Z}_t^{n,\vv}:={\bf X}_t^{n}-{\bf X}_t^{n,n,\vv}$ with
 $
 {\bf X}_t^{n}: =\big(X_t^{1}, \cdots,X_t^{n}\big)$ and ${\bf X}_t^{n,n,\vv}: =\big(X_t^{1,n,\vv}, \cdots,X_t^{n,n,\vv}\big).$
 Roughly speaking,   in \eqref{EW3} the asymptotic coupling by reflection is employed for small jumps, and the synchronous coupling is explored  for large jumps.

  The main result in this part is presented as follows.

 \begin{proposition}\label{pro2}
Fix $n\ge1$ and $T>0$. Let $({\bf X}^{n}_{[0,T]},{\bf X}^{n,n,\vv}_{[0,T]})_{\vv>0}=(({\bf X}^{n}_t)_{t\in[0,T]},({\bf X}^{n,n,\vv}_t)_{t\in[0,T]})_{\vv>0}$ be the process determined by \eqref{EW3} such that the initial value   $({\bf X}^{n}_0,{\bf X}^{n,n,\vv}_0)_{\vv>0}$ satisfies all  properties mentioned above.
Under  $({\bf A}_1)$,
 $({\bf X}^{n}_{[0,T]},{\bf X}^{n,n,\vv}_{[0,T]})_{\vv>0}$ has a  weakly convergent  subsequence
  such that the corresponding weak limit process is the coupling process of ${\bf X}^{n}_{[0,T]}$ and ${\bf X}^{n,n}_{[0,T]}$,
 where  ${\bf X}^{n,n}_{[0,T]}:=({\bf X}^{n,n}_t)_{t\in[0,T]}$ with
 ${\bf X}_t^{n,n}: =\big(X_t^{1,n}, \cdots,X_t^{n,n}\big)$ for any $t\ge0.$
\end{proposition}

In order to examine the tightness of $({\bf X}^{n,n,\vv}_{[0,T]})_{\vv>0}$, it is primary to demonstrate that $({\bf X}_{[0,T]}^{n,n,\vv})_{\vv>0} $
 has a  uniform
  moment with regard to the parameter  $\vv$, which is claimed in the subsequent  lemma.

\begin{lemma}\label{uniform}
Fix $n\ge1$ and $T>0$.
Suppose Assumption $({\bf A1})$ holds and further $ \E| X_0^{1,n}|<\8.$
Then,  there is a constant $C_{T}>0$
$($which is independent of $n$$)$
such that for any $\vv>0$,
\begin{equation}\label{E7}
 \E\Big(\sup_{0\le t\le T}|{\bf X}_t^{n,n,\vv}|\Big)\le C_{T}n\big(1+\E| X_0^{1,n}|\big).
\end{equation}
\end{lemma}

\begin{proof}
As in the proof of Proposition \ref{pro1}, we still write $V(x)=(1+|x|^2)^{\frac{1}{2}}, x\in\R^d.$
Note that  for any $x,y,z\in\R^d$,
\begin{align*}
V(x+y\I_{\{|z|\le1\}})+V(x+y\I_{\{|z|>1\}})-V(x)=V(x+y)-V(x).
\end{align*}
Then, applying It\^o's formula yields that
\begin{align*}
&\d V(X_t^{i,n,\vv})\\
&=\<\nn V(X_t^{i,n,\vv}),b(X_t^{i,n,\vv},\hat\mu_t^{n,\vv})\>\,\d t  +\d M_t^{i,n,\vv}\\
 &\quad   +\int_{\{|z|<\frac{1}{2|\si|}|Z_t^{i,n,\vv}|\}}\big[V(X_t^{i,n,\vv}+\si\Pi_{\vv,t}z)-V(X_t^{i,n,\vv})
 -\si\<\nn V(X_t^{i,n,\vv}),\Pi_{\vv,t}z\>\I_{\{|z|<1\}}\big]\,\nu(\d z)\d t\\
 &\quad   +\int_{\{|z|\ge\frac{1}{2|\si|}|Z_t^{i,n,\vv}|\}}\big[V(X_t^{i,n,\vv}+\si z)-V(X_t^{i,n,\vv})
 -\si\<\nn V(X_t^{i,n,\vv}), z\>\I_{\{|z|<1\}}\big]\,\nu(\d z)\d t\\
 &\quad   +\int_{\{|z|<\frac{1}{2|\si_0|}|Z_t^{i,n,\vv}|\}}\big[V(X_t^{i,n,\vv}+\si_0\Pi_{\vv,t}z)
 -V(X_t^{i,n,\vv})-\si_0\<\nn V(X_t^{i,n,\vv}),\Pi_{\vv,t}z\>\I_{\{|z|<1\}}\big]\,\nu^0(\d z)\d t\\
 &\quad   +\int_{\{|z|\ge\frac{1}{2|\si_0|}|Z_t^{i,n,\vv}|\}}\big[V(X_t^{i,n,\vv}+\si_0z)
 -V(X_t^{i,n,\vv})-\si_0\<\nn V(X_t^{i,n,\vv}),z\>\I_{\{|z|<1\}}\big]\,\nu^0(\d z)\d t,
\end{align*}
where $\Pi_{\vv,t}:=\Pi_{\vv,d}({\bf Z}_t^{n,\vv})
$,
and
\begin{equation}\label{E9}
\begin{split}
M^{i,n,\vv}_t
: &=  \bigg(\int_0^t\int_{\{|z|<\frac{1}{2|\si|}|Z_s^{i,n,\vv}|\}}\big[V(X_s^{i,n,\vv}+\si\Pi_{\vv,s}z )-V(X_s^{i,n,\vv})\big]\,\wt N^i(\d z,\d s)
\\
 &\quad   +\int_0^t\int_{\{|z|\ge\frac{1}{2|\si|}|Z_s^{i,n,\vv}|\}}\big[V(X_s^{i,n,\vv}+\si z)-V(X_s^{i,n,\vv})\big]\,\wt N^i(\d z,\d s)
 \bigg)\\
 &\quad   +\bigg(\int_0^t\int_{\{|z|<\frac{1}{2|\si_0|}|Z_s^{i,n,\vv}|\}}\big[V(X_s^{i,n,\vv}+\si_0\Pi_{\vv,s}z )-V(X_s^{i,n,\vv})\big]\,\wt N^0(\d z,\d s)\\
 &\quad   +\int_0^t\int_{\{|z|\ge\frac{1}{2|\si_0|}|Z_s^{i,n,\vv}|\}}\big[V(X_s^{i,n,\vv}+\si_0z)-V(X_s^{i,n,\vv})\big]\,\wt N^0(\d z,\d s)\bigg)\\
 &=: \Theta^{i,n,\vv}_t+\bar\Theta^{i,n,\vv}_t.
\end{split}
\end{equation}
Next, by repeating the strategy to derive \eqref{ET4} and using the fact that
$\|\Pi_{\vv,t}\|_{\rm HS}^2\le d$, there exists a constant $c_1>0$ such that
\begin{align*}
\d V(X_t^{i,n,\vv})
 \le c_1\big(1+|X_t^{i,n,\vv}|+\hat{\mu}_t^{n,\vv}(|\cdot|)\big)\,\d t+  \d M_t^{i,n,\vv}.
\end{align*}

 Apparently, one has
\begin{equation*}
\begin{split}
\Theta^{i,n,\vv}_t&=\int_0^t\int_{\{|z|<1\wedge(\frac{1}{2|\si|}|Z_s^{i,n,\vv}|)\}}\big[V(X_s^{i,n,\vv}+\si\Pi_{\vv,s}z )-V(X_s^{i,n,\vv})\big]\,\wt N^i(\d z,\d s)\\
&\quad+\int_0^t\int_{\{ 1\wedge(\frac{1}{2|\si|}|Z_s^{i,n,\vv}|)\le|z|\le \frac{1}{2|\si|}|Z_s^{i,n,\vv}|\}}\big[V(X_s^{i,n,\vv}+\si\Pi_{\vv,s}z )-V(X_s^{i,n,\vv})\big]\, N^i(\d z,\d s)\\
&\quad+\int_0^t\int_{\{ 1\wedge(\frac{1}{2|\si|}|Z_s^{i,n,\vv}|)\le|z|\le \frac{1}{2|\si|}|Z_s^{i,n,\vv}|\}}\big[V(X_s^{i,n,\vv}+\si\Pi_{\vv,s}z )-V(X_s^{i,n,\vv})\big]\,\nu(\d z) \d s\\
&\quad+\int_0^t\int_{\{|z|\ge1\vee(\frac{1}{2|\si|}|Z_s^{i,n,\vv}|)\}}\big[V(X_s^{i,n,\vv}+\si z)-V(X_s^{i,n,\vv})\big]\,  N^i(\d z,\d s)\\
&\quad+\int_0^t\int_{\{|z|\ge1\vee(\frac{1}{2|\si|}|Z_s^{i,n,\vv}|)\}}\big[V(X_s^{i,n,\vv}+\si z)-V(X_s^{i,n,\vv})\big]\,\nu(\d z) \d s\\
&\quad+\int_0^t\int_{\{\frac{1}{2|\si|}|Z_s^{i,n,\vv}|\le  |z|<1\vee(\frac{1}{2|\si|}|Z_s^{i,n,\vv}|)\}}\big[V(X_s^{i,n,\vv}+\si z)-V(X_s^{i,n,\vv})\big]\,\wt N^i(\d z,\d s).
\end{split}
\end{equation*}
Thereafter, applying the Burkholder-Davis-Gundy inequality (see, for instance,  \cite[Theorem 1]{MR}) and utilizing the fact that the random measure $N^i(\d z,\d s)$
is nonnegative, we deduce from $\|\nn V\|_\8\le1$ and $\|\Pi_{\vv,t}\|_{\rm HS}^2\le d$ that there exist  constants  $c_2,c_3 >0$ such that
\begin{align}\label{ET5}
&\E\Big(\sup_{0\le s\le t } |\Theta^{i,n,\vv}_s|\Big)\nonumber\\
&\le  c_2\E\bigg(\int_0^t\int_{\{|z|<(\frac{1}{2|\si|}|Z_s^{i,n,\vv}|)\wedge1\}}\big|V (X_s^{i,n,\vv}+\si\Pi_{\vv,s}z)-V (X_s^{i,n,\vv})\big|^2\,\nu(\d z)\d s\bigg)^{1/2}\nonumber\\
&\quad   +c_2\E\bigg(\int_0^t\int_{\{(\frac{1}{2|\si|}|Z_s^{i,n,\vv}|)\wedge1\le |z|<\frac{1}{2|\si|}|Z_s^{i,n,\vv}|\}}\big|V (X_s^{i,n,\vv}+\si \Pi_{\vv,s}z)-V (X_s^{i,n,\vv})\big|\,\nu(\d z)\d s\bigg)\nonumber\\
&\quad+c_2\E\bigg(\int_0^t\int_{\{|z|\ge1\vee(\frac{1}{2|\si|}|Z_s^{i,n,\vv}|)\}}\big|V(X_s^{i,n,\vv}+\si z)-V(X_s^{i,n,\vv})\big|\,\nu(\d z) \d s\bigg)\\
&\quad+c_2\E\bigg(\int_0^t\int_{\{\frac{1}{2|\si|}|Z_s^{i,n,\vv}| \le  |z|<1\vee(\frac{1}{2|\si|}|Z_s^{i,n,\vv}|)\}}\big|V(X_s^{i,n,\vv}+\si z)-V(X_s^{i,n,\vv})\big|^2\,\nu(\d z) \d s\bigg)^{\frac{1}{2}}\nonumber\\
&\le  c_2(1+\sqrt{d})\,|\si|\sqrt{t} \bigg( \int_{\{|z|<1\}} |z|^2\,\nu(\d z) \bigg)^{1/2}
      +c_2(1+\sqrt{d}) |\si| t\int_{\{|z|\ge1\}}|z|\,\nu(\d z) \nonumber\\
&\le  c_3(\sqrt{t}+t),\nonumber
\end{align}
where in the second inequality we used the fact that the events $\{\frac{1}{2|\si|}|Z_s^{i,n,\vv}|\wedge1\le |z|<\frac{1}{2|\si|}|Z_s^{i,n,\vv}|\}$
and  $\{\frac{1}{2|\si|}|Z_s^{i,n,\vv}| \le  |z|<1\vee(\frac{1}{2|\si|}|Z_s^{i,n,\vv}|)\}$ are empty
in case   the events  $\{\frac{1}{2|\si|}|Z_s^{i,n,\vv}|\le 1\}$ and $\{1\le\frac{1}{2|\si|}|Z_s^{i,n,\vv}|\}$
 take  place, respectively, and the last inequality holds true due to   \eqref{vv0}. Next, by following the same line to deduce \eqref{ET5}, we have
\begin{equation*}
 \E\Big(\sup_{0\le s\le t } |\bar\Theta^{i,n,\vv}_s|\Big)  \le  c_4(\sqrt{t}+t).
\end{equation*}
Accordingly, there is a constant $c_5>0$ such that
\begin{align*}
\frac{1}{n}\sum_{i=1}^n\E\Big(\sup_{0\le s\le t }| X_s^{i,n,\vv}|\Big)   &\le \frac{1}{n}\sum_{i=1}^n\E\Big(\sup_{0\le s\le t }V ( X_s^{i,n,\vv})\Big)\\
&\le 1+\frac{1}{n}\sum_{i=1}^n\E|X_0^{i,n,\vv}| +c_5(\sqrt{T}+T)+\frac{c_5}{n}\sum_{i=1}^n\int_0^t  \E \sup_{0\le u\le s }|X_u^{i,n,\vv}|\,\d s.
\end{align*}
Finally, the assertion \eqref{E7}  follows immediately from
 Gr\"onwall's inequality and by noting that
\begin{align*}
\E\Big(\sup_{0\le t\le T}|{\bf X}_t^{n,n,\vv}|\Big)\le \sum_{i=1}^n\E\Big(\sup_{0\le t\le T }| X_t^{i,n,\vv}|\Big).
\end{align*} The proof is therefore complete.
\end{proof}

\begin{lemma}\label{tight}
Fix $n\ge1$ and $T>0$. Suppose Assumption $({\bf A}_1)$ holds and further $\E| X_0^{1,n}|<\8.$  Then,    $({\bf X}^{n,n,\vv}_{[0,T]})_{\vv>0}$ is tight.
\end{lemma}

\begin{proof}
Below, we fix $n\ge1,T>0$,   and  write $D([0,T];\R^d)$
as the space of functions $f:[0,T]\rightarrow\R^d$
that are right-continuous and have left-hand limits. It is
 obvious that  ${\bf X}^{n,n,\vv}_{[0,T]}\in D([0,T];\R^d)$ for any  $\vv>0.$ As we know,
one of the   methods to examine  tightness of the $D([0,T];\R^d)$-valued stochastic processes is Aldous's criterion; see, for example,  \cite[Theorem 1]{Aldous}. Accordingly,    to show the tightness of
 $({\bf X}^{n,n,\vv}_{[0,T]})_{\vv>0}$, it is sufficient to demonstrate the following statements:
\begin{enumerate}
\item[(i)] for each $t\in[0,T]$, $({\bf X}^{n,n,\vv}_t)_{\vv>0}$ is tight;

\item[(ii)] ${\bf X}^{n,n,\vv}_{\tau_\vv+\delta_\vv}- {\bf X}^{n,n,\vv}_{\tau_\vv}\to 0$ in probability as $\vv\to0$, where, for each $\vv>0$, $ \tau_\vv\in[0,T] $ is a stopping time and  $\delta_\vv\in[0,1]$ is a constant such that  $\delta_\vv\to0$ as $\vv\to0.$
\end{enumerate}
Indeed,  the statement (i) is provable  by taking Lemma \ref{uniform} and Chebyshev’s inequality into account.  So, in the sequel, it remains to   verify the statement  (ii).

From \eqref{EW3}, it is easy to see that for any $\beta>0, $
\begin{align*}
\P\big(\big|{\bf X}^{n,n,\vv}_{\tau_\vv+\delta_\vv}- {\bf X}^{n,n,\vv}_{\tau_\vv}\big|\ge \beta\big)
&\le \sum_{i=1}^n \Bigg(\P\bigg(\int_{\tau_\vv}^{\tau_\vv+\delta_\vv}\big|b(  X_s^{i,n,\vv},\hat{ \mu}_s^{n,\vv})\big|\,\d s\ge \frac{\beta}{5n}\bigg) \\
&\quad\quad\quad\quad +\P\bigg( |\si| \bigg|\int_{\tau_\vv}^{\tau_\vv+\delta_\vv}\int_{\{|z|<\frac{1}{2|\si|}|Z_s^{i,n,\vv}|\}}\Pi_{\vv,s}\cdot z\,\ol N^i(\d z,\d s)\bigg|\ge \frac{\beta}{5n}\bigg)\\
&\quad\quad\quad\quad +\P\bigg( |\si| \bigg|\int_{\tau_\vv}^{\tau_\vv+\delta_\vv}\int_{\{|z|\ge\frac{1}{2|\si|}|Z_s^{i,n,\vv}|\}}z\,\ol N^i(\d z,\d s)\bigg|\ge \frac{\beta}{5n}\bigg)\\
&\quad\quad\quad\quad +\P\bigg( |\si_0| \bigg|\int_{\tau_\vv}^{\tau_\vv+\delta_\vv}\int_{\{|z|<\frac{1}{2|\si_0|}|Z_s^{i,n,\vv}|\}}\Pi_{\vv,s}\cdot z\,\ol N^0(\d z,\d s)\bigg|\ge \frac{\beta}{5n}\bigg)\\
&\quad\quad \quad\quad+\P\bigg( |\si_0| \bigg|\int_{\tau_\vv}^{\tau_\vv+\delta_\vv}\int_{\{|z|\ge\frac{1}{2|\si_0|}|Z_s^{i,n,\vv}|\}}z\,\ol N^0(\d z,\d s)\bigg|\ge \frac{\beta}{5n}\bigg)\Bigg)\\
&=:\sum_{i=1}^n\sum_{j=1}^5 \Gamma^{j,\vv}_i.
\end{align*}

By leveraging Chebyshev's inequality and \eqref{E7}, it follows that  for any $R_0>0$,
\begin{equation*}
\P\bigg(\sup_{0\le t\le T+1}|{\bf X}_t^{n,n,\vv}|\ge R_0\bigg)\le \frac{1}{R_0 }C_{T+1}n\big(1+\E| X_0^{1,n}|
 \big).
\end{equation*}
This implies that,  for any $\vv_0>0,$   there exists an  $R_0^*=R_0^*(\vv_0)>0$ such that
\begin{align}\label{E10}
\P\bigg(\sup_{0\le t\le T+1}|{\bf X}_t^{N,N,\vv}|\ge R_0^*\bigg)\le \vv_0.
\end{align}
With the quantity $R_0^*$ above at hand, we define the following
  the stopping time
\begin{equation*}
\tau_0^{n,\vv}=\inf\big\{t\ge0: |{\bf X}_t^{n,n,\vv}|> R_0^*\big\}.
\end{equation*}
Subsequently,  we find from \eqref{E2}  that
\begin{align*}
 \Gamma^{1,\vv}_i&\le\P\bigg(\int_{\tau_\vv}^{\tau_\vv+\delta_\vv}\big|b(  X_s^{i,n,\vv},\hat{ \mu}_s^{n,\vv})-b(  X_s^{i,n,\vv},\delta_{\bf0})\big|\,\d s\ge \frac{\beta}{10n}\bigg)\\
 &\quad+\P\bigg(\int_{\tau_\vv}^{\tau_\vv+\delta_\vv}\big|b(  X_s^{i,n,\vv},\delta_{\bf0})\big|\,\d s\ge \frac{\beta}{10n}\bigg)\\
 &\le \P\bigg( \int_{\tau_\vv}^{\tau_\vv+\delta_\vv}\mathbb W_1(\hat{ \mu}_s^{n,\vv},\delta_{\bf0}) \,\d s\ge \frac{\beta}{10nL_2}\bigg)+ \P\big( \tau_0^{n,\vv}\le T+1\big)\\
 &\quad+\P\bigg(\int_{\tau_\vv}^{\tau_\vv+\delta_\vv}\big|b(  X_s^{i,n,\vv},\delta_{\bf0})\big|\,\d s\ge \frac{\beta}{10n},\tau_0^{n,\vv}> T+1\bigg)\\
 &\le  \P\bigg( \frac{1}{n}\sum_{j=1}^n\int_{\tau_\vv}^{\tau_\vv+\delta_\vv}|X_s^{j,n,\vv}| \,\d s\ge \frac{\beta}{10nL_2}\bigg)+ \P\bigg(\sup_{0\le t\le T+1}|{\bf X}_t^{n,n,\vv}|\ge R_0^*\bigg)\\
 &\quad+\P\bigg(\int_{\tau_\vv}^{\tau_\vv+\delta_\vv}\I_{[0,\tau_0^{n,\vv})}(s)\big|b(  X_s^{i,n,\vv},\delta_{\bf0} )\big|\,\d s\ge \frac{\beta}{10n} \bigg).
\end{align*}
Thereby,  $ \lim_{\vv\downarrow0}\Gamma^{1,\vv}_i=0 $ is available by recalling that $b(\cdot,\delta_{\bf0})$ is   locally bounded on $\R^d$ (see Assumption $({\bf A}_1)$) and making use of   \eqref{E7}, \eqref{E10} as well as  $\lim_{\vv\downarrow0}\delta_\vv=0$.

Next, applying Chebyshev's inequality followed by It\^o's isometry yields that
\begin{align*}
 \Gamma^{2,\vv}_i&\le\P\bigg( \bigg|\int_{\tau_\vv}^{\tau_\vv+\delta_\vv}\int_{\{|z|<\frac{1}{2|\si|}|Z_s^{i,n,\vv}|\}}\Pi_{\vv,s}\cdot z\I_{\{|z|\le1\}}\tt N^i(\d z,\d s)\bigg|\ge \frac{\beta}{10n |\si|}\bigg)\\
 &\quad+\P\bigg( \bigg|\int_{\tau_\vv}^{\tau_\vv+\delta_\vv}\int_{\{|z|<\frac{1}{2|\si|}|Z_s^{i,n,\vv}|\}}\Pi_{\vv,s}\cdot z\I_{\{|z|>1\}}  N^i(\d z,\d s)\bigg|\ge \frac{\beta}{10n |\si|}\bigg)\\
 &\le \frac{100n^2\si^2}{\beta^2}\E\bigg|\int_{\tau_\vv}^{\tau_\vv+\delta_\vv}\int_{\{|z|<\frac{1}{2|\si|}|Z_s^{i,n,\vv}|\}}\Pi_{\vv,s}\cdot z\I_{\{|z|\le1\}}\tt N^i(\d z,\d s)\bigg|^2\\
 &\quad+\frac{10n |\si|}{\beta}\E\bigg|\int_{\tau_\vv}^{\tau_\vv+\delta_\vv}\int_{\{|z|<\frac{1}{2|\si|}|Z_s^{i,n,\vv}|\}}\Pi_{\vv,s}\cdot z\I_{\{|z|>1\}}  N^i(\d z,\d s)\bigg|\\
 &\le \frac{100n^2\si^2}{\beta^2}\E\bigg(\int_{\tau_\vv}^{\tau_\vv+\delta_\vv}\int_{\{|z|\le 1\}}|\Pi_{\vv,s}\cdot z|^2\nu (\d z) \d s\bigg) \\
 &\quad+\frac{10n |\si|}{\beta}\E\bigg(\int_{\tau_\vv}^{\tau_\vv+\delta_\vv}\int_{\{|z|>1\}}|\Pi_{\vv,s}\cdot z|  \nu(\d z)\d s\bigg).
\end{align*}
This, along with $\|\Pi_{\vv,t}\|_{\rm HS}^2\le d$, \eqref{vv0},  and $\lim_{\vv\downarrow0}\delta_\vv=0$, leads to
$ \lim_{\vv\downarrow0}\Gamma^{2,\vv}_i=0 $. In the same way, we can conclude that
$
\sum_{j=2}^5\lim_{\vv\downarrow0}\Gamma^{j,\vv}_i=0
$.  Consequently, based on the previous analysis,
 the statement (ii) is verifiable.
\end{proof}

Before we move forward to start the proof of Proposition  \ref{pro1},  we introduce some additional notation.  Denote $\mathscr D_\8= D([0,\8);(\R^{d})^n)$ the family  of
 functions $\psi:[0,\8)\to (\R^{d})^n$
  that are right-continuous and have left-hand limits,
  and write $\pi: \mathscr D_\8
  \to (\R^{d})^n$  as
  the projection operator, which is defined     by $\pi_t\psi=\psi(t)$ for $\psi\in \mathscr D_\8
  $
  and $t\ge0.$ In addition, we set
 $\mathcal F_t:=\si(\pi_s:s\le t)$, i.e.,  the $\sigma$-algebra on $\mathscr D_\8$
 induced by the projections $(\pi_s)_{s\in[0,t]}$.

\  \

With Lemma  \ref{tight} at hand,  the proof of Proposition \ref{pro2} can be finished.
\begin{proof}[Proof of Proposition  \ref{pro2}]
Lemma \ref{tight}, besides the Prohorov theorem, implies that, for fixed $n\ge1$ and $T>0$,   $({\bf X}^{n}_{[0,T]},{\bf X}^{n,n,\vv}_{[0,T]})_{\vv>0}$ has a weakly convergent  subsequence, written as  $({\bf X}^{n}_{[0,T]},{\bf X}^{n,n,\vv_l}_{[0,T]})_{l\ge 0} $,
with the  corresponding  weak limit, denoted by $({\bf X}^{n}_{[0,T]},\tilde{{\bf X}}^{n,n}_{[0,T]}) $, in which   $(\vv_l)_{l\ge0}$ is a sequence satisfying  $\lim_{l\to \infty}\vv_l=0.$
In order to  demonstrate  that
$({\bf X}^{n}_{[0,T]},\tilde{{\bf X}}^{n,n}_{[0,T]}) $ is the desired  coupling process associated with   ${\bf X}^{n}_{[0,T]}$ and ${\bf X}^{n,n}_{[0,T]}$, it is sufficient to examine  $
\mathscr L_{\tilde{{\bf X}}^{n,n}}=\mathscr L_{{\bf X}^{n,n}},
$ where $\mathscr L_{\tilde{{\bf X}}^{n,n}}$ and $\mathscr L_{{\bf X} ^{n,n}}$ stands respectively  for  the infinitesimal generators of $(\tilde{\bf X}_t^{n,n})_{t\ge0}$ and  $({\bf X}_t^{n,n})_{t\ge0}$.
Note that
for $f\in C_b^2((\R^d)^n)$ and ${\bf x}:=(x^1,\cdots,x^n)\in  (\R^d)^n$,
\begin{equation*}
\begin{split}
\big(\mathscr L_{{\bf X} ^{n,n}}f\big)({\bf x}) =\sum_{i=1}^n\Big(&\<\nn_if( {\bf x} ),  b(x^i,\hat \mu_{{\bf x}}^n)\>
+\int_{\R^d}\big(f({\bf x}+\si s_i(z)  )-f({\bf x})-\si\<\nn_if({\bf x}),  z\>\I_{\{|z|<1\}}\big)\,\nu(\d z)\\
&+\int_{\R^d}\big(f({\bf x}+\si_0 s_i(z) )-f({\bf x})-\si_0\<\nn_if({\bf x}), z\>\I_{\{|z|<1\}}\big)\,\nu^0(\d z)\Big),
\end{split}
\end{equation*}
where $\hat\mu_{\bf x}^n:=\frac{1}{n}\sum_{j=1}^n\delta_{x^j}$, $\nn_i$ is the first-order gradient operator with respect to the $x^i$-component, and $s_i(z):=({\bf0},\cdots, z,\cdots,{\bf 0})$, i.e., the $i$-th component of $({\bf0},\cdots, {\bf 0},\cdots,{\bf 0})$ is replaced by the vector $z\in\R^d.$

For any $f\in C^2_b((\R^{d})^n)$, define the quantity
\begin{align*}
M_t^{n,f}=f(\tilde{{\bf X}}^{n,n}_t)-f(\tilde{{\bf X}}^{n,n}_0)-\int_0^t\big(\mathscr L_{{\bf X} ^{n,n}}f\big)(\tilde{{\bf X}}^{n,n}_s)\,\d s.
\end{align*}
Provided that for any $t\ge s\ge0$ and $\mathcal F_s$-measurable bounded continuous functional $F:\mathscr D_\8\to\R$,
\begin{equation}\label{E11}
\E\big(M_t^{n,f}F(\tilde{{\bf X}}^{n,n})\big)=\E\big(M_s^{n,f}F(\tilde{{\bf X}}^{n,n})\big),
\end{equation}
that is to say, $(M_t^{n,f})_{t\ge0}$ is a martingale with respect to the filtration $(\mathcal F_t)_{t\ge0}$,
we then can conclude  that
 $\mathscr L_{\tilde{{\bf X}}^{n,n}}=\mathscr L_{{\bf X}^{n,n}}$  by the aid of  the weak uniqueness of \eqref{EW2}.

In the sequel,  we aim at proving  \eqref{E11}.
For  ${\bf x}\in(\R^{d})^n$, let $\mathscr L^{n,\vv}_{{\bf x}}$ be the infinitesimal generator of $({\bf X}_t^{n,n,\vv})_{t\ge0}$ based on the prerequisite that
  the Markov process $( {\bf X}_t^{n,n })_{t\ge0}$ is given. Via a direct calculation, the relationship between $\mathscr L^{n,\vv}_{{\bf x}}$ and $\mathscr L_{{\bf X} ^{n,n}}$ can be given as below:
  for given ${\bf x}\in(\R^{d})^n$ and  any   $f\in C^2_b((\R^{d})^n)$ and ${\bf y}\in(\R^{d})^n$,
\begin{equation}\label{E12}
\begin{split}
&\big(\mathscr L^{n,\vv}_{{\bf x}}f\big)({\bf y})\\
&=\big(\mathscr L_{{\bf X} ^{n,n}}f\big)({\bf y})\\
&\quad  -\sum_{i=1}^n\int_{\{|z|<\frac{1}{2|\si|}|z^i|\}}\big(f({\bf y}+\si s_i(z))-f({\bf y})-\si\<\nn_if({\bf y}),  z\>\I_{\{|z|<1\}}\\
&\quad  -\big(f({\bf y}+\si  s_i(\Pi_{\vv,d}
({\bf x}-{\bf y})z))-f({\bf y})-\si\<\nn_if({\bf y}),\Pi_{\vv,d}
({\bf x}-{\bf y}) z\>\I_{\{|z|<1\}}\big)\big)\,\nu(\d z)\\
&\quad  -\sum_{i=1}^n\int_{\{|z|<\frac{1}{2|\si_0|}|z^i|\}}\big(f({\bf y}+\si_0  s_i(z))-f({\bf y})-\si_0\<\nn_if({\bf y}), z\>\I_{\{|z|<1\}}\\
&\quad   -\big(f({\bf y}+\si_0 s_i(\Pi_{\vv,d}
({\bf x}-{\bf y})z) )-f({\bf y})-\si_0\<\nn_if({\bf y}),\Pi_{\vv,d}
({\bf x}-{\bf y}) z\>\I_{\{|z|<1\}}\big)\big)\,\nu^0(\d z)\\
&=:\big(\mathscr L_{{\bf X} ^{n,n}}f\big)({\bf y})-\big(\mathscr L^{n,\vv,\nu}_{{\bf x}}f\big)({\bf y})-\big(\mathscr L^{n,\vv,\nu^0}_{{\bf x}}f\big)({\bf y}),
\end{split}
\end{equation}
 where $z^i:=x^i-y^i.$

Via It\^o's formula, for $f\in C_b^2((\R^{d})^n)$,  we know that
 $(M_t^{n,f,\vv_l})_{t\ge0}$,   defined in the manner of
\begin{align*}
M_t^{n,f,\vv_l} =f({\bf X}^{n,n,\vv_l}_t)-f({\bf X}^{n,n,\vv_l}_0)-\int_0^t\big(\mathscr L^{n,\vv_l}_{{\bf X}^{n}_s}f\big)({\bf X}^{n,n,\vv_l}_s)\,\d s
\end{align*}
is a martingale with respect to $(\mathcal F_t)_{t\ge0}$ so   for any $t\ge s\ge0$ and $\mathcal F_s$-measurable bounded continuous functional $F:\mathscr D_\8\to\R$,
\begin{equation}\label{E13}
\E\big(M_t^{n,f,\vv_l}F( {\bf X}^{n,n,\vv_l})\big)=\E\big(M_s^{n,f,\vv_l}F( {\bf X}^{n,n,\vv_l})\big).
\end{equation}
Apparently,   with the help of  \eqref{E12}, $(M_t^{n,f,\vv_l})_{t\ge0}$ can be reformulated in the form below: for any $t\ge0,$
\begin{align*}
M_t^{n,f,\vv_l}&=f({\bf X}^{n,n,\vv_l}_t)-f({\bf X}^{n,n,\vv_l}_0)-\int_0^t(\mathscr L_{{\bf X} ^{n,n}}f)({\bf X}^{n,n,\vv_l}_s)\,\d s
+\int_0^t \big(\mathscr L^{n,\vv_l,*}_{{\bf X}^{n}_s}f\big)({\bf X}^{n,n,\vv_l}_s)\,\d s.
\end{align*}
Thereby, the  assertion \eqref{E11} can be available   by invoking   \eqref{E12},  applying  the dominated convergence theorem and exploiting the statements to be claimed   that
\begin{equation}\label{E14}
\lim_{\vv\to 0}\big(\mathscr L^{N,\vv,\nu}_{{\bf x}}f\big)({\bf y})=0\quad\mbox{ and } \quad \lim_{\vv\to 0}\big(\mathscr L^{N,\vv,\nu^0}_{{\bf x}}f\big)({\bf y})=0.
\end{equation}

Once  the assertion  $\lim_{\vv\to 0}\big(\mathscr L^{N,\vv,\nu}_{{\bf x}}f\big)({\bf y})=0$ is done,  the proof of  $\lim_{\vv\to 0}\big(\mathscr L^{N,\vv,\nu^0}_{{\bf x}}f\big)({\bf y})=0$ can be established in the same manner. Therefore, in the following analysis, we focus merely on the proof of the former one.
Hereinafter, for brevity, we set for given ${\bf x},{\bf y}\in(\R^d)^n$,
\begin{align*}
\Phi_i(\vv,z)&:=f({\bf y}+\si s_i(z))-f({\bf y})-\si\<\nn_if({\bf y}), z\>\I_{\{|z|\le1\}}\\
&\quad\,\, -\big(f({\bf y}+\si s_i(\Pi_{\vv,d}
({\bf x}-{\bf y})z))-f({\bf y})-\si\<\nn_if({\bf y}),\Pi_{\vv,d}
({\bf x}-{\bf y}) z\>\I_{\{|z|\le1\}}\big).
\end{align*}
Notice that for given  ${\bf y}\in (\R^d)^n$ and any  ${\bf z}\in(\R^d)^n$,
\begin{align*}
 f({\bf y}+{\bf z})-f({\bf y} )
 &= \<\nn f({\bf y}),{\bf z }\>
 +\int_0^1\int_0^s \<\nn^2 f({\bf y}+u{\bf z}){\bf z} , {\bf z}\>\,\d u\,\d s.
\end{align*}
Whence,
we find that
\begin{align*}
&\int_{\{|z|<\frac{1}{2|\si|}|z^i|,|z|\le1\}}\Phi_i(\vv,z)\,\nu(\d z)\\
&=\si^2\int_0^1\int_0^s\int_{\{|z|<\frac{1}{2|\si|}|z^i|,|z|\le1\}}\Big(\<\nn^2_i f({\bf y}+u \si s_i(z))z , z\> \\
&\quad- \<\nn^2_i f({\bf y}+u \si   s_i( \Pi_{\vv,d}
({\bf x}-{\bf y}) z)) \Pi_{\vv,d}
({\bf x}-{\bf y}) z ,  \Pi_{\vv,d}
({\bf x}-{\bf y}) z\>\Big)\nu(\d z)\d u\d s.
\end{align*}
In terms of the definition of $h_\vv$, it is ready to see that
\begin{align*}
\lim_{\vv\to0}\Pi_{\vv,d}
({\bf x}-{\bf y})=
\begin{cases}
\Pi_d({\bf x}-{\bf y}):=I_d-2 {\bf n}(\phi({\bf x}-{\bf y}))\otimes{\bf n}(\phi({\bf x}-{\bf y})),\quad \mbox{ if }  \rho({\bf x}-{\bf y})\neq0,\\
I_d,~~~~~~~~~~~~~~~~~~~~~~~~~~~~~~~~~~~~~~~~~~~~~~~~~~~~~~~~~~~~~~ \mbox{ if } \rho({\bf x}-{\bf y})=0.
\end{cases}
\end{align*}
This enables us to derive that  for any $u\in[0,1],$
\begin{align*}
&\lim_{\vv\to0}\<\nn^2_i f({\bf y}+u \si   s_i( \Pi_{\vv,d}({\bf x}-{\bf y}) z)) \Pi_{\vv,d}({\bf x}-{\bf y}) z ,  \Pi_{\vv,d}({\bf x}-{\bf y}) z\> \\
&=
\begin{cases}
\<\nn^2_i f({\bf y}+u \si   s_i( \Pi_d({\bf x}-{\bf y}) z)) \Pi_d({\bf x}-{\bf y}) z ,  \Pi_d({\bf x}-{\bf y}) z\>,\quad \mbox{ if }  \rho({\bf x}-{\bf y})\neq0,\\
\<\nn^2_i f({\bf y}+u \si   s_i(   z))   z ,  z\>,~~~~~~~~~~~~~~~~~~~~~~~~~~~~~~~~~~~~~~~~~~~ \mbox{ if }  \rho({\bf x}-{\bf y})=0.
\end{cases}
\end{align*}
Subsequently, applying the dominated convergence theorem  and taking the rotationally invariant property of $\nu(\d z)$ yields that
\begin{align*}
\lim_{\vv\to0}\int_{\{|z|<\frac{1}{2|\si|}|z^i|,|z|\le1\}}\Phi_i(\vv,z)\,\nu(\d z)=0.
\end{align*}

On the other hand, by virtue of
\begin{align*}
& \int_{\{|z|<\frac{1}{2|\si|}|z^i|,|z|>1\}}\Phi_i(\vv,z)\,\nu(\d z)\\
&=\si\int_0^1\int_{\{|z|<\frac{1}{2|\si|}|z^i|,|z|>1\}}\<\nn_if({\bf y}+s\si s_i(z)),z\>\,\nu(\d z)\,\d s\\
 &\quad-\si\int_0^1\int_{\{|z|<\frac{1}{2|\si|}|z^i|,|z|>1\}}\<\nn_if({\bf y}+s\si s_i(\Pi_{\vv,d}({\bf x}-{\bf y})z)),\Pi_{\vv,d}({\bf x}-{\bf y})z\>\,\nu(\d z)\,\d s,
 \end{align*}
 along with the  dominated convergence theorem  and   the rotationally invariant property of $\nu(\d z)$ once more, it follows that
 \begin{align*}
\lim_{\vv\to0}\int_{\{|z|<\frac{1}{2|\si|}|z^i|,|z|>1\}}\Phi_i(\vv,z)\,\nu(\d z)=0.
\end{align*}
 In the end, we conclude that the establishment $\lim_{\vv\to 0}\big(\mathscr L^{N,\vv,\nu}_{{\bf x}}f\big)({\bf y})=0$ is available.
\end{proof}

\section{Proof of Theorem \ref{thm1}}\label{sec3}

This section is devoted to accomplishing the proof of Theorem \ref{thm1}.
In particular, we herein are  concentrated merely in the $1$-dimensional SDE \eqref{EW0}, where $\si,\si_0\neq0$.
The corresponding interpretation why we work on the $1$-dimensional setting will be detailed in Remark \ref{remark}.

To proceed, we show that
  $((X_t^i)_{t>0})_{1\le i\le n} $  determined by    \eqref{EW1} has finite first-order moment  in an infinite-time horizon.

 \begin{lemma}\label{lem1}
Assume that $({\bf H}_1)$ holds with  $\lambda_2>\lambda_3$, and suppose further that $(X_0^i)_{1\le i\le n}$ are i.i.d. $\mathscr F_0$-measurable
random variables such that $\E|X_0^1|<\8$. Then,
 there is a constant $C_0>0$  such that for all $1\le i\le n$,
\begin{align}\label{E16}
 \sup_{t\ge0}\E|X_t^i| \le \E|X_0^1| + C_0.
\end{align}
\end{lemma}

\begin{proof}
In order to establish  \eqref{E16}, it only necessitates to amend the associated details to derive \eqref{EX}, which is concerned with the first-order moment estimate in a finite horizon. Below, we just stress the associated  distinctness.

From   (${\bf H}_1$), it is easy to see that  for all $x\in\R$ and $\mu\in \mathscr P_1(\R)$,
\begin{equation}\label{ET8}
x b(x,\mu)
\le  (\lambda_1+\lambda_2)|x|^2\I_{\{|x|\le\ell_0\}}-\lambda_2|x|^2+(\lambda_3\mu(|\cdot|)+|b(0,\delta_{0})|)|x|.
\end{equation}
Below, we set $\lambda_*:=\lambda_2-\lambda_3 $ and
 write  $V(x)=(1+|x|^2)^{\frac{1}{2}}, x\in\R.$
By applying  It\^o's formula and using \eqref{ET1}, \eqref{ET2} as well as \eqref{ET8}, there exists a constant $c_0>0$ such that
\begin{align*}
\d \big(\e^{  \lambda_* t}V(X_t^i)\big)
&\le\e^{  \lambda_* t}\bigg( \lambda_* V(X_t^i) +\frac{X_t^i}{V (X_t^i)}b(X_t^i, \mu_t^i)   +c_0\int_{\R^d}(|z|^2\wedge|z|)(\nu + \nu^0)(\d z)\bigg)\,\d t+\d M_t^i\\
&\le \e^{  \lambda_*t}\big(  (\lambda_*-\lambda_2) V(X_t^i)  +\lambda_3\mu^i_t(|\cdot|)+c_1 \big)\,\d t +\d M_t^i\\
&\le \e^{  \lambda_*t}\big(  (\lambda_*-\lambda_2) V(X_t^i)  +\lambda_3 \E V(X_t^i) +\lambda_3(\mu^i_t(|\cdot|)-\E|X_t^i|)+c_1\big)\,\d t +\d M_t^i,
\end{align*}
where  $(M_t^i)_{t\ge0}$ is a martingale, and
 $$c_1:= c_0\int_{\R^d}(|z|^2\wedge|z|)(\nu + \nu^0)(\d z)+\lambda_2+(\lambda_1+\lambda_2)\ell_0+|b(0,\delta_{0})|.$$
Subsequently, we deduce from $\E^0\mu^i_t(|\cdot|)=\E\mu^i_t(|\cdot|)=\E|X_t^i|$
(see \eqref{ET9}) and $\lambda_*=\lambda_2-\lambda_3$ that
\begin{equation*}
\E V(X_t^i)
 \le \E V( X_0^i)+ c_1/ {\lambda_*}.
\end{equation*}
This, together with the hypothesis that
$(X_0^i)_{1\le i\le n}$ are i.i.d. $\mathscr F_0$-measurable
random variables,
  implies   the desired  assertion \eqref{E16}.
\end{proof}

Recall that the concrete expression of
the function $\rho:(\R^d)^n\to[0,\8)$ involved in Subsection \ref{section2.2}  is undetermined. From now on, we shall choose
$$\rho({\bf x})= \|{\bf x}\|_{1}:=\frac{1}{n}\sum_{j=1}^n|x^j|,\quad {\bf x}\in\R^n $$
so,
for the setting $d=1$,
$
\Pi_{\vv}({\bf x}):= \Pi_{\vv,1}({\bf x})
=1-2h_\vv( \|{\bf x}\|_{1}),
  {\bf x}\in \R^n.
$
With the previous function $\rho(\cdot)$ at hand, the issue on
the uniform-in-time conditional PoC for the $1$-dimensional McKean-Vlasov SDE \eqref{EW0} can be treated  via the asymptotic coupling by reflection.

\begin{proposition}\label{pro3}
Assume  that  $({\bf H}_1)$-$({\bf H}_3)$
hold and
suppose that
 \begin{align}\label{la0}
\lambda_0:=\lambda^*-\lambda_3\e^{\Lambda_1}>0 \quad \mbox{ with } \quad \lambda^*:=\min\{\lambda_1\e^{-\Lambda_2},\lambda_2\e^{-\Lambda_1}\},
\end{align}
  where $\lambda_1,\lambda_2,\lambda_3>0$ are introduced in $({\bf H}_1)$,
and
\begin{equation*}
\Lambda_1:= \lambda_1 \int_0^{2\ell_0}\frac{r}{  F_{\si,\si_0}(r)}\d r, \quad    \quad \Lambda_2:= \lambda_1 \int_0^{\ell_0}\frac{r}{  F_{\si,\si_0}(r)}\d r
\end{equation*}
with  the function $F_{\si,\si_0}$ being given in $({\bf H}_3)$.
Then,
there exists a  constant $C_0>0$
$($which is independent of $n\ge1$$)$ such that for any $t\ge0,$
\begin{equation}\label{PoC}
 \E \|{\bf Z}_t^{n,\vv}\|_1 \le  C_0\e^{-\lambda_0 t}  \E \|{\bf Z}_0^{n,\vv}\|_1+C_0\Big(\frac{1}{n}\big(1+\E |X_0^1|\big)+\varphi(n)+   \vv \Big),
\end{equation}
where ${\bf Z}_t^{n,\vv}:=(Z_t^{1,n,\vv},\cdots,Z_t^{n,n,\vv})$  with $Z_t^{i,n,\vv}:=X_t^i-X_t^{i,n,\vv}$,
 and $\varphi(\cdot)$ is given in $({\bf H}_2)$.
\end{proposition}

\begin{proof}
Below, we split the proof into three parts since the detailed  proof is a little bit lengthy, and fix  $1\le i\le n$.

{\it$(i)$  Stochastic differential inequality solved by the radial process}.
 Notice from   \eqref{EW3}   that
\begin{align*}
\d Z_t^{i,n,\vv}=&\big(b(X_t^i,\mu_t^i) -b(X_t^{i,n,\vv},\hat\mu_t^{n,\vv})\big)\,\d t
+2\si\int_{\{|z|<\frac{1}{2|\si|}|Z_t^{i,n,\vv}|\}}h_\vv\big(\|{\bf Z}_t^{n,\vv}\|_1\big)z\,\ol{N}^i(\d t,\d z)\\
&+2\si_0\int_{\{|z|<\frac{1}{2|\si_0|}|Z_t^{i,n,\vv}|\}}h_\vv\big(\|{\bf Z}_t^{n,\vv}\|_1\big)z\,\ol{N}^0(\d t,\d z)
\end{align*}
and that for any $a,x,z\in\R,$
\begin{align*}
|x+az\I_{\{|z|\le
1\}}|+|x+az\I_{\{|z|>1  \}}|-2|x|=|x+az|-2|x|.
\end{align*}
Thus,  applying It\^o's formula yields  that
\begin{equation}\label{E17}
\begin{split}
\d |Z_t^{i,n,\vv}|&=\frac{Z_t^{i,n,\vv}}{|Z_t^{i,n,\vv}|}\big(b(X_t^i,\mu_t^i) -b(X_t^{i,n,\vv},\hat\mu_t^{n,\vv})\big)\I_{\{|Z_t^{i,n,\vv}|\neq0\}}\,\d t \\
&\quad+\int_{\{|z|<\frac{1}{2|\si|}|Z_t^{i,n,\vv}|\}} \Lambda^{i, \vv}({\bf Z}_t^{n,\vv},\si,z) \,\nu(\d z)\d t\\
&\quad+\int_{\{|z|<\frac{1}{2|\si_0|}|Z_t^{i,n,\vv}|\}}\Lambda^{i, \vv}({\bf Z}_t^{n,\vv},\si_0,z)\,\nu^0(\d z)\d t+\d M_t^{i,n,\vv}\\
&=:\frac{Z_t^{i,n,\vv}}{|Z_t^{i,n,\vv}|}\big(b(X_t^i,\mu_t^i) -b(X_t^{i,n,\vv},\hat\mu_t^{n,\vv})\big)\I_{\{|Z_t^{i,n,\vv}|\neq0\}}\,\d t+\big(\phi_t^{i,n,\vv} + \bar\phi_t^{i,n,\vv}\big)\,\d t+ \d M_t^{i,n,\vv},
\end{split}
\end{equation}
where  for ${\bf x}\in\R^n$ and $u,z\in\R$,
\begin{align*}
\Lambda^{i, \vv}({\bf x},u,z):=\big|x^i+2u h_\vv(\|{\bf x} \|_1)z\big|-|x^i| -\frac{x^i}{|x^i|}2u h_\vv(\|{\bf x}  \|_1)z\I_{\{|z|\le1\}}\I_{\{|x^i|\neq0\}}
\end{align*}
and
\begin{align*}
\d M_t^{i,n,\vv}
:=&\int_{\{|z|<\frac{1}{2|\si|}|Z_t^{i,n,\vv}|\}}\big(\big|Z_t^{i,n,\vv}+2\si h_\vv\big(\|{\bf Z}_t^{n,\vv}\|_1\big)z\big|
    -|Z_t^{i,n,\vv}|\big)\,\widetilde{N}^i(\d t,\d z)\\
&+\int_{\{|z|<\frac{1}{2|\si_0|}|Z_t^{i,n,\vv}|\}}\big(\big|Z_t^{i,n,\vv}+2\si_0 h_\vv\big(\|{\bf Z}_t^{n,\vv}\|_1\big)z\big|
    -|Z_t^{i,n,\vv}|\big)\,\widetilde{N}^0(\d t,\d z).
\end{align*}

In the sequel, we write
\begin{align*}
\tilde\mu_t^n =\frac{1}{n}\sum_{j=1}^n\delta_{X_t^j} \quad \mbox{ and } \quad   \tilde{\mu}_t^{n,-i} =\frac{1}{n-1}\sum_{j=1:j\neq i}^n\delta_{X_t^j}.
\end{align*}
Trivially, we have
\begin{align*}
\tilde\mu_t^n =\frac{1}{n}\sum_{j=1}^n\delta_{X_t^j}=\frac{n-1}{n}\tilde{\mu}_t^{n,-i}+\frac{1}{n}\delta_{X_t^i}.
\end{align*}
Subsequently, the following fact (see e.g. \cite[(3.16)]{BW}) that for $\mu\in\mathscr P_1(\R )$ and $x\in\R $,
\begin{align*}
\mathbb W_1\Big(\frac{n-1}{n}\mu+\frac{1}{n}\delta_x,\mu\Big)\le \frac{1}{n}\big(|x|+\mu(|\cdot|)\big)
\end{align*}
enables us to derive that
\begin{align*}
\mathbb W_1(\tilde\mu_t^n,\tilde{\mu}_t^{n,-i})\le \frac{1}{n}\big(|X_t^i|+ \tilde{\mu}_t^{n,-i}(|\cdot|)\big).
\end{align*}
Next,  by means of $({\bf H}_1)$ and $({\bf H}_2)$, along with the triangle inequality,
it holds that
\begin{equation}\label{E18}
\begin{split}
 &\frac{Z_t^{i,n,\vv}}{|Z_t^{i,n,\vv}|}\big(b(X_t^i,\mu_t^i)-b(X_t^{i,n,\vv},\hat\mu_t^{n,\vv})\big)\I_{\{Z_t^{i,n,\vv}\neq0\}}\\
 &\le \frac{Z_t^{i,n,\vv}}{|Z_t^{i,n,\vv}|}\big(b(X_t^{i},\tilde\mu_t^n)-b(X_t^{i,n,\vv},\hat\mu_t^{n,\vv})\big)\I_{\{Z_t^{i,n,\vv}\neq0\}}\\
 &\quad +|b(X_t^i,\tilde\mu_t^{n,-i})-b(X_t^{i},\tilde\mu_t^n)|+|b(X_t^{i},\mu_t^i)-b(X_t^{i},\tilde\mu_t^{n,-i})|\\
 &\le (\lambda_1+\lambda_2)|Z_t^{i,n,\vv}|\I_{\{|Z_t^{i,n,\vv}|\le\ell_0\}}-\lambda_2|Z_t^{i,n,\vv}|+ \lambda_3 \tilde\mu_t^n(|\cdot|) +J_i({\bf X}_t^n),
 \end{split}
\end{equation}
where
$$J_i({\bf X}_t^n):=\frac{\lambda_3}{n}\big(|X_t^i|+ \tilde{\mu}_t^{n,-i}(|\cdot|)\big)+ \big|b(X_t^i,\mu_t^i)-b(X_t^i,\tilde{\mu}_t^{n,-i})\big|.$$

As we know, the utmost importance  is that the quadratic variation process of the
associated radial process vanishes when the (asymptotic) coupling by reflection is applied to SDEs driven by Brownian motion.
Analogously to the aforementioned fact,  it is extremely important to necessitate $\phi_t^{i,n,\vv}=\bar\phi_t^{i,n,\vv}=0$, where the terms $\phi_t^{i,n,\vv}, \bar\phi_t^{i,n,\vv}$ play the similar role as the quadratic variation process corresponding to the Brownian motion case. For ${\bf x}\in\R^n$ and $0\neq u\in\R$, in case of
$|z|< |x^i|/{2|u|}$, it follows from
$h_\vv\in[0,1] $ that
\begin{align}\label{ET11}
x^i+2u h_\vv\big(\|{\bf x} \|_1\big)z\ge x^i-2 |u|\cdot|z|\ge0\quad\quad \mbox{ if } \quad x^i\ge0,
\end{align}
and
\begin{align}\label{ET12}
x^i+2\si h_\vv\big(\|{\bf x} \|_1\big)z\le x^i+2|u| \cdot|z|\le0\quad\quad \mbox{ if } \quad x^i<0.
\end{align}
So, we arrive at $\phi_t^{i,n,\vv}=0$ and $ \bar\phi_t^{i,n,\vv}=0$ in case of  $ |Z_t^{i,n,\vv}|/{2|\si|}\le 1$ and $ |Z_t^{i,n,\vv}|/{2|\si_0|}\le 1$, respectively. On the other hand,  via the rotationally invariant property of $\nu(\d z)$, for $x\in\R$ and $0\neq u\in\R,$
\begin{align*}
\int_{\{|z|<\frac{1}{2|u|}|x |\}}z\I_{\{|z|>1\}}\nu(\d z)=0\quad \mbox{ if } \quad |x |/{2|u|}> 1.
\end{align*}
 Whence, we also have $\phi_t^{i,n,\vv}=0$ and $\bar \phi_t^{i,n,\vv}=0$ once  $ |Z_t^{i,n,\vv}|/{2|\si|}> 1$ and $ |Z_t^{i,n,\vv}|/{2|\si_0|}> 1$, separately. So,   $\phi_t^{i,n,\vv}=\bar \phi_t^{i,n,\vv}=0$ is available. Based on the preceding analysis,
we derive that
\begin{equation}\label{ET10}
\begin{split}
\d |Z_t^{i,n,\vv}|
&\le \big((\lambda_1+\lambda_2)|Z_t^{i,n,\vv}|\I_{\{|Z_t^{i,n,\vv}|\le\ell_0\}}-\lambda_2|Z_t^{i,n,\vv}|\big)\I_{\{|Z_t^{i,n,\vv}|\neq0\}}\,\d t\\
&\quad +\big(\lambda_3 \tilde\mu_t^n(|\cdot|) +J_i({\bf X}_t^n)\big)\,\d t+\d M_t^{i,n,\vv}.
\end{split}
\end{equation}

{\it $(ii)$  Stochastic differential inequality solved by the composition of the radial process and the distance function}.
Define the following function:
\begin{equation}\label{f}
f(r) =\begin{cases}
             \int_0^{r}\e^{-g_*(s)}\,\d s, &\quad r\in [0,2\ell_0], \\
             f(2\ell_0)+f'(2\ell_0)(r-2\ell_0), & \quad r\in [2\ell_0,\infty),\end{cases}
\end{equation}
where
$$g_*(r)
=  \lambda_1 \int_0^r\frac{s}{F_{\si,\si_0}(s)}\,\d s,\quad r\in[0,2\ell_0]$$
and
 $F_{\si,\si_0}(\cdot)$  are given in $({\bf H}_3)$.
Applying It\^o's formula and taking   \eqref{ET10}, \eqref{ET11} as well as \eqref{ET12} into consideration gives that for  $\lambda_0$ given in \eqref{la0},
\begin{align*}
&\d \big(\e^{\lambda_0 t}f(|Z_t^{i,n,\vv}|)\big)\\
&\le\d\bar M_t^{i,n,\vv}+\e^{\lambda_0 t}\Big(\lambda_0 f(|Z_t^{i,n,\vv}|) +f'(|Z_t^{i,n,\vv}|)\big((\lambda_1+\lambda_2)|Z_t^{i,n,\vv}|\I_{\{|Z_t^{i,n,\vv}|\le\ell_0\}}-\lambda_2|Z_t^{i,n,\vv}|\big) \\
&\qquad\qquad\qquad\qquad\,\, +f'(|Z_t^{i,n,\vv}|)\big(\lambda_3 \tilde\mu_t^n(|\cdot|) +J_i({\bf X}_t^n)\big)\Big)\,\d t\\
&\quad+\e^{\lambda_0 t} \int_{ \{|z|<\frac{1}{2|\si|}|Z_t^{i,n,\vv}| \}}\Big(f\big(\big|Z_t^{i,n,\vv}+2\si h_\vv\big(\|{\bf Z}_t^{n,\vv}\|_1\big)z\big|\big)-f(|Z_t^{i,n,\vv}|)\\
&\qquad\qquad\qquad\qquad\qquad\quad\,\,  -2\si\frac{Z_t^{i,n,\vv}}{|Z_t^{i,n,\vv}|}f'(|Z_t^{i,n,\vv}|) h_\vv\big(\|{\bf Z}_t^{n,\vv}\|_1\big)z\I_{\{|z|\le1\}}\I_{\{|Z_t^{i,n,\vv}|\neq0\}}\Big)\,\nu(\d z)\d t\\
&\quad  +\e^{\lambda_0 t}\int_{ \{|z|<\frac{1}{2|\si_0|}|Z_t^{i,n,\vv}| \}}\Big(f\big(\big|Z_t^{i,n,\vv}+2\si_0 h_\vv\big(\|{\bf Z}_t^{n,\vv}\|_1\big)z\big|\big)-f(|Z_t^{i,n,\vv}|)\\
&\qquad\qquad\qquad\qquad\qquad\quad\,\,  -2\si_0\frac{Z_t^{i,n,\vv}}{|Z_t^{i,n,\vv}|}f'(|Z_t^{i,n,\vv}|) h_\vv\big(\|{\bf Z}_t^{n,\vv}\|_1\big)z\I_{\{|z|\le1\}}\I_{\{|Z_t^{i,n,\vv}|\neq0\}}\Big)\,\nu^0(\d z)\d t,
\end{align*}
where $(\bar M_t^{i,n,\vv})_{t\ge0}$ is a martingale. By virtue of the rotational invariance of $\nu(\d z)$ and the odd property of the mapping $z\mapsto z\I_{\{|z|\le1\}}$, it follows that
\begin{align*}
 \d \big(\e^{\lambda_0 t}f(|Z_t^{i,n,\vv}|)\big)
&\le\e^{\lambda_0 t}\Big(\lambda_0 f(|Z_t^{i,n,\vv}|) +f'(|Z_t^{i,n,\vv}|)\big((\lambda_1+\lambda_2)|Z_t^{i,n,\vv}|\I_{\{|Z_t^{i,n,\vv}|\le\ell_0\}}-\lambda_2|Z_t^{i,n,\vv}|\big) \\
&\qquad\quad\,\, +f'(|Z_t^{i,n,\vv}|)\big(\lambda_3 \tilde\mu_t^n(|\cdot|) +J_i({\bf X}_t^n)\big)\Big)\,\d t\\
&\quad+\frac{1}{2}\e^{\lambda_0 t} \int_{ \{|z|<\frac{1}{2|\si|}|Z_t^{i,n,\vv}| \}}\Upsilon^{i,n,\vv}(t,\si,z)\,\nu(\d z)\d t\\
&\quad  +\frac{1}{2}\e^{\lambda_0 t}\int_{ \{|z|<\frac{1}{2|\si_0|}|Z_t^{i,n,\vv}| \}}\Upsilon^{i,n,\vv}(t,\si_0,z)\,\nu^0(\d z)\d t +\d\bar M_t^{i,n,\vv},
\end{align*}
where for $u,z\in\R $ and $t\ge0,$
\begin{align*}
\Upsilon^{i,n,\vv}(t,u,z):=f\big(\big|Z_t^{i,n,\vv}+2u h_\vv(\|{\bf Z}_t^{n,\vv}\|_1)z\big|\big)+f\big(\big|Z_t^{i,n,\vv}-2u h_\vv(\|{\bf Z}_t^{n,\vv}\|_1)z\big|\big)-2f(|Z_t^{i,n,\vv}|).
\end{align*}
For ${\bf x}\in\R^n$ and  $u,z\in\R$,
note that the hypothesis that   $|z|\le \frac{1}{2|u|}|x^i|$ with  $u\neq0,$
 $x^i\ge0$ and $x^i<0$
implies respectively that
\begin{align*}
 x^i\pm2u
 h_\vv\big(\|{\bf x} \|_1\big)z  \ge0 \quad \mbox{ and } \quad x^i\pm2u
 h_\vv\big(\|{\bf x} \|_1\big)z<0.
 \end{align*}
Thereby, in case of $|z|\le \frac{1}{2|u|}|Z_t^{i,n,\vv}|$ for $u\neq0,$
$\Upsilon^{i,n,\vv}(t,u,z)$
can be rewritten as below:
\begin{align*}
\Upsilon^{i,n,\vv}(t,u,z)
&=f\big( |Z_t^{i,n,\vv}|+2|u| h_\vv(\|{\bf Z}_t^{n,\vv}\|_1)|z|  \big)+f\big( |Z_t^{i,n,\vv}|-2|u| h_\vv(\|{\bf Z}_t^{n,\vv}\|_1)|z|  \big)-2f(|Z_t^{i,n,\vv}|).
\end{align*}
Next, since   $[0,\8)\ni r\mapsto f'(r)$ is decreasing,
the mean value theorem implies that
\begin{align*}
f(r+\delta)+f(r-\delta)-2f(r)\le0, \quad 0\le \delta\le r.
\end{align*}
Correspondingly,   $\Upsilon^{i,n,\vv}(t,u,z)\le0$ provided $|z|\le  \frac{1}{2|u|}|Z_t^{i,n,\vv}|$ for $u\neq0.$ Moreover,
the fact  (see e.g.  \cite[Lemma 4.1]{LMW}) that
\begin{align*}
f(r+\delta)+f(r-\delta)-2f(r)\le f''(r)\delta^2,\quad 0\le \delta\le r\le\ell_0
\end{align*}
(also owing  to $g''_*(r)\le0$, $g^{(3)}_*(r)\ge0$ and $g^{(4)}_*(r)\le0$ for $r\in(0,2\ell_0]$),
and  the hypothesis that $|Z_t^{i,n,\vv}|\le\ell_0$ and  $|z|\le  \frac{1}{2|u|}|Z_t^{i,n,\vv}|$ for $u\neq0,$
imply  that
\begin{align*}
\Upsilon^{i,n,\vv}(t,u,z)\le4 f''(|Z_t^{i,n,\vv}|) |u|^2 h_\vv(\|{\bf Z}_t^{n,\vv}\|_1)^2|z|^2.
\end{align*}
As a consequence, due to  $f''(r)<0, r\le \ell_0,$
we deduce
from \eqref{EEE} that
 that
\begin{align*}
 \d \big(\e^{\lambda_0 t}f(|Z_t^{i,n,\vv}|)\big)
&\le\d\bar M_t^{i,n,\vv} +\e^{\lambda_0 t}\big(\lambda_0 f(|Z_t^{i,n,\vv}|) +\psi(|Z_t^{i,n,\vv}|)  \big) \,\d t\\
&\quad +\e^{\lambda_0 t}f'(|Z_t^{i,n,\vv}|)\big(\lambda_3 \tilde\mu_t^n(|\cdot|) +J_i({\bf X}_t^n)\big)\,\d t+  \e^{\lambda_0 t}   \varphi^{\vv,i}( {\bf Z}_t^{n,\vv} )  \d t,
\end{align*}
where for any $r\ge0, $
\begin{align*}
\psi(r):=f'(r)\big((\lambda_1+\lambda_2)r\I_{\{r\le\ell_0\}}-\lambda_2r\big)+2 f''(r) F_{\si,\si_0}(r)  \I_{\{r\le\ell_0\}}
\end{align*}
and
\begin{align*}
\varphi^{\vv,i}({\bf x}):= 2   f''(|x^i|)  (h_\vv(\|{\bf x} \|_1)^2-1)  F_{\si,\si_0}(|x^i|)\I_{\{|x^i|\le\ell_0\}}.
\end{align*}

{\it $(iii)$  Establishment of \eqref{PoC}}.
Owing to $g'_*(r) =\frac{ \lambda_1r}{ F_{\si,\si_0}(r)} $ for all $r\in(0,2\ell_0],$   it is easy to see that
\begin{align*}
 \psi(r)=
-\lambda_1r   \e^{-g_*(r)} ,\quad    r\le \ell_0 \quad\mbox{ and } \quad   \psi(r)=
-\lambda_2f'(r)r, \quad  r>\ell_0.
\end{align*}
Whence, we arrive at
\begin{align}\label{ET13}
 \psi(r)\le - \lambda^*r,\quad r\ge0.
\end{align}
Additionally, by invoking $({\bf H}_2)$ and Lemma \ref{lem1}, there exists a constant $c_0>0$ such that
\begin{align*}
\lambda_3  \E\tilde\mu_t^n(|\cdot|) +\frac{1}{n}\sum_{i=1}^n\E J_i({\bf X}_t^n)&\le  \lambda_3 \E \|{\bf Z}_t^{n,\vv}\|_1+\frac{\lambda_3}{n^2}\sum_{i=1}^n\Big( \E|X_t^i|+\frac{1}{n-1}\sum_{j=1:j\neq i}^n\E|X_t^j| \Big)+ \varphi(n)\\
&\le  \lambda_3\E \|{\bf Z}_t^{n,\vv}\|_1+\frac{c_0}{n}(1+\E|X_0^1| ) + \varphi(n).
\end{align*}
This, besides $f'\le1 $, $f'(2\ell_0)r\le f(r) $ as well as  \eqref{ET13}, yields that
\begin{equation}\label{ET14}
\begin{split}
\frac{1}{n}\sum_{i=1}^n\E f(|Z_t^{i,n,\vv}|)&\le \frac{\e^{-\lambda_0 t}}{n}\sum_{i=1}^n\E f(|Z_0^{i,n,\vv}|) + \frac{c_0}{n\lambda_0}(1+\E|X_0^1| )  + \frac{\varphi(n)}{\lambda_0}\\
&\quad + \frac{1}{n}\sum_{i=1}^n\int_0^t\e^{-\lambda_0(t-s)}\varphi^{\vv,i}( {\bf Z}_s^{n,\vv} )\,\d s.
\end{split}
\end{equation}
Furthermore, by means of $f''(r)=-g'_*(r)\e^{-g_*(r)}$ and  $g'_*(r) =\frac{ \lambda_1r}{ F_{\si,\si_0}(r)} $ for $r\in[0,\ell_0],$   we obtain from $h_\vv\in[0,1] $ that
\begin{equation}\label{ET15}
\begin{split}
 \frac{1}{n} \sum_{i=1}^n\varphi^{\vv,i}({\bf x})&= 2  \big(1-h_\vv(\|{\bf x} \|_1)^2 \big) \frac{1}{n}\sum_{i=1}^n  g'_*(|x^i|)  \e^{-g_*(|x^i|)} F_{\si,\si_0}(|x^i|)\I_{\{|x^i|\le\ell_0\}}\\
  &=2 \lambda_1 \big(1-h_\vv(\|{\bf x} \|_1)^2 \big) \frac{1}{n}\sum_{i=1}^n \e^{-g_*(|x^i|)}|x^i|\\
  &\le 4 \lambda_1 \big(1-h_\vv(\|{\bf x} \|_1)  \big)\|{\bf x}\|_1\\
  &\le 8\lambda_1\vv,
\end{split}
\end{equation}
where in the last display we used the fact that $(1-h_\vv(r))r\le 2\vv$  for all $r\ge0.$ At length, the assertion \eqref{PoC} follows from \eqref{ET14}, \eqref{ET15}, as well as  $f'(2\ell_0)r\le f(r)\le r, r\ge0. $
\end{proof}

Before we proceed, we make an  additional  comment.
\begin{remark}\label{remark}
 Note that \eqref{E17} is still valid for the high dimensional case (i.e., $d\ge2$).  Nevertheless, for this setting, it is a tough task to
 verify  $\phi_t^{i,n,\vv}=\bar \phi_t^{i,n,\vv}=0$, which plays a crucial role in establishing \eqref{PoC}. Therefore, in the present work, we focus merely on the $1$-dimensional case.
\end{remark}

In the sequel, we provide an illustrative example on $g_*(\cdot)$ given in $({\bf H}_3)$.  

\begin{example}\label{exa}
Let $\nu (\d z)=\frac{c_*}{|z|^{1+\alpha}}$ and $\nu^0(\d z)=\frac{c^*}{|z|^{1+\beta}}$ for some constants $c_*,c^*>0$ and $\alpha,\beta\in(1,2)$.
By virtue of $\alpha,\beta\in(1,2)$, it is ready to see that \eqref{vv0} is fulfilled. A direct calculation shows that
 for any $r\ge0,$
\begin{align*}
 2c_*\si^2\int_{\{0\le z <\frac{r}{2|\si|} \}}z^{1-\alpha} \d z +2c^*\si^2_0\int_{\{0\le  z <\frac{r}{2|\si_0|} \}}z^{1-\beta} \d z =\frac{2^{\alpha-1}c_*|\si|^\alpha r^{2-\alpha}}{2-\alpha} +\frac{2^{\beta-1}c^*|\si_0|^\beta r^{2-\beta}}{2-\beta}.
 \end{align*}
Note that for
fixed
$\theta\in(0,1)$,
\begin{align*}
 |\si|^\alpha r^{2-\alpha}  + |\si_0|^\beta r^{2-\beta}\ge C_\theta ( |\si|^\alpha +|\si_0|^\beta   )r^{2- \theta },\quad r\in[0,2\ell_0],
\end{align*}
where
$
C_\theta:=(2\ell_0)^{ \theta-\alpha}\wedge (2\ell_0)^{ \theta-\beta}
$ for $\ell_0\ge1.$
Below, we take
\begin{align*}
F_{\si,\si_0}(r)=C_1 ( |\si|^\alpha +|\si_0|^\beta   )r^{2- \theta },\quad r\in[0,2\ell_0],
\end{align*}
where $
C_1:=C_\theta\big(\frac{2^{\alpha-1}c_*  }{2-\alpha}\wedge\frac{2^{\beta-1}c^*   }{2-\beta}\big).
$
Subsequently, we have
\begin{align}\label{EE-}
  g_*(r)=\frac{\lambda_1r^\theta}{C_1\theta(|\si|^\alpha+|\si_0|^\beta)},\quad r\in[0,2\ell_0].
\end{align}
Due to $\theta\in(0,1)$, it is easy to see  that $g_*'(r)>0$, $g_*''(r)<0$, $g_*'''(r)>0$ as well as $g_*^{(4)}(r)<0$ for all $r\in(0,2\ell_0].$ Additionally,
 we notice from \eqref{EE-} that   $(|\si|,|\si_0|)\mapsto \Lambda_1=\Lambda_1(|\si|,|\si_0|)$ and  $(|\si|,|\si_0|)\mapsto \Lambda_2=\Lambda_2(|\si|,|\si_0|)$
 are decreasing in two respective variables. So, the bigger intensity of the independent noise and the common noise can enhance the associated convergence rate.
\end{example}

With all the preparations above at hand, we move on to conduct the proof of Theorem   \ref{thm1}.
\begin{proof}[Proof of Theorem $\ref{thm1}$]
In retrospect,
  $((X_t^{i})_{t>0})_{1\le i\le n}$ and  $((\bar X_t^{i})_{t>0})_{1\le i\le n}$ are governed by
  \eqref{EW1} with respective initial value $(X_0^{i})_{1\le i\le n}$ and  $(\bar X_0^{i})_{1\le i\le n}$, and
 $((X_t^{i,n})_{t>0})_{1\le i\le n}$  is the solution to \eqref{EW2} with the   initial value $(\bar X_0^{i})_{1\le i\le n}$.

For $\Gamma\in \mathscr C(\mathscr L_{\mu_0},\mathscr L_{\bar\mu_0})$, there exists a  measure-valued random variable $(m_0,\bar m_0)$ such that $\mathscr L_{(m_0,\bar m_0)}=\Gamma$ so $\mathscr L_{m_0}=\mathscr L_{\mu_0}$ and $\mathscr L_{\bar m_0}=\mathscr L_{\bar\mu_0}$. Subsequently, there is a measure-valued random variable $\xi$ such that
\begin{align*}
\mathbb W_1(m_0,\bar m_0)=\int_{\R\times \R}|x-y|\xi(\d x,\d y).
\end{align*}
In the following analysis,   $(X_0^{i},\bar X_0^i)_{1\le i\le n}$ are set  to be identically distributed and  mutually  independent and satisfy $\mathscr L_{(X_0^{i},\bar X_0^i)|\mathscr F_0^0}=\xi$. Correspondingly,  we derive that
\begin{align*}
\E|X_0^{i}-\bar X_0^{i}|=\E\big(\E\big(|X_0^{i}-\bar X_0^{i}|\big|\mathscr F_0^0\big)\big)&=\E\bigg(\int_{\R\times \R}|x-y|\mathscr L_{(X_0^{i},\bar X_0^{i})|\mathscr F_0^0}(\d x,\d y)\bigg)\\
&=\E\mathbb W_1(m_0,\bar m_0)=\int_{\mathscr P(\R)\times\mathscr P(\R)}\mathbb W_1(\mu,\nu)\Gamma(\d \mu,\d \nu).
\end{align*}
 Whence, we arrive at
\begin{align}\label{EY}
\E|X_0^{i}-\bar X_0^{i}| =\mathcal W_1(\mathscr L_{\mu_0},\mathscr L_{\bar \mu_0}),
\quad i=1,\cdots,n.
\end{align}

Note from Proposition \ref{pro4} that for any given $T>0 $ and all $i=1,\cdots, n,$
\begin{align*}
\P^0\big(\mu_t^i=\mu_t\mbox{ for all } t\in[0,T]\big)=1.
\end{align*}
Then, by invoking  the triangle inequality, it is easy to see that for all $t>0$ and $i=1,\cdots,n,$
\begin{equation}\label{e:ppff1}\begin{split}
 \mathcal W_1\big(\mathscr L_{\mu_t},\mathscr L_{\bar{\mu}_t}\big)&=\mathcal W_1\big(\mathscr L_{\mu_t^i},\mathscr L_{\bar{\mu}_t^i}\big)\\
&\le \E^0\mathbb W_1(\mu_t^i,\bar{\mu}_t^i)\\
&\le \E^0 \big(\E^1\mathbb W_1(\mu_t^i,\tilde\mu_t^{n })\big)+\E^0\big(\E^1\mathbb W_1(\tilde\mu_t^{n },\hat\mu_t^{n})\big)\\
&\quad+\E^0\big(\E^1\mathbb W_1(\hat\mu_t^{n},\bar\mu_t^{n})\big) +\E^0\big(\E^1\mathbb W_1(\bar{\mu}_t^i,\bar\mu_t^{n} )\big)\\
&= \E \mathbb W_1(\mu_t^i,\tilde\mu_t^{n })+\E\mathbb W_1(\tilde\mu_t^{n },\hat\mu_t^{n}) +\E\mathbb W_1(\hat\mu_t^{n},\bar\mu_t^{n}) +\E\mathbb W_1(\bar{\mu}_t^i,\bar\mu_t^{n} )\\
&=:\Gamma_1(t,n)+\Gamma_2(t,n)+\Gamma_3(t,n)+\Gamma_4(t,n),
\end{split}\end{equation}
where
\begin{align*}
\tilde\mu_t^{n}: =\frac{1}{n}\sum_{j=1}^n\delta_{X_t^{j }},\quad \bar\mu_t^{n}: =\frac{1}{n}\sum_{j=1}^n\delta_{\bar X_t^{j }}\quad \mbox{ and } \quad \hat\mu_t^{n}:=\frac{1}{n}\sum_{j=1}^n\delta_{X_t^{j,n}}.
\end{align*}

From   Proposition \ref{pro1},  we deduce  that
\begin{align*}
\lim_{n\to\8}\big(\Gamma_1(t,n)+\Gamma_4(t,n)\big)=0.
\end{align*}
Since    $(\bar X_t^{i },X_t^{i,n })_{1\le i\le n}$ are identically distributed, it follows that
\begin{align*}
\Gamma_3(t,n)\le\frac{1}{n}\sum_{j=1}^n\E|\bar X_t^{j }-X_t^{j,n }|=\E |\bar X_t^{1 }-X_t^{1,n}|.
\end{align*}
Subsequently,  applying Proposition \ref{pro1} once more leads to
$
\lim_{n\to\8}\Gamma_3(t,n)=0.
$
By Fatou's lemma, we have
\begin{align*}
 \E \mathbb W_1(\tilde\mu_t^{n},\hat\mu_t^{n})&\le \frac{1}{n}\sum_{j=1}^n\E |X_t^{j}-X_t^{j,n}| \le \frac{1}{n}\sum_{j=1}^n\liminf_{m\to\8}\E\big(m\wedge|X_t^{j}-X_t^{j,n}|\big).
\end{align*}
Thereafter, by leveraging Proposition \ref{pro2} and Fatou's lemma, we deduce that
\begin{align*}
 \E \mathbb W_1(\tilde\mu_t^{n},\hat\mu_t^{n}) &\le \frac{1}{n}\sum_{j=1}^n\liminf_{m\to\8}\liminf_{\vv\to0}\E\big(m\wedge|X_t^{j }-X_t^{j,n,\vv}|\big)\\
 &\le \frac{1}{n}\sum_{j=1}^n \liminf_{\vv\to0}\E |X_t^{j }-X_t^{j,n,\vv}| \\
 &\le \liminf_{\vv\to0} \E \|{\bf Z}_t^{n,\vv}\|_1,
\end{align*}
where $((X_t^i)_{t\ge0},(X_t^{i,n,\vv})_{t\ge0})_{1\le i\le n}$ solves \eqref{EW3}.
Obviously, there is a constant $\lambda_3^*>0$ such that $\lambda_0$, defined in \eqref{la0}, is positive when $\lambda_3\in(0,\lambda_3^*].$
Next, an application of Proposition \ref{pro3}  yields that
\begin{align*}
\Gamma_2(t,n)
  &\le \liminf_{\vv\to0}\bigg( C_0\e^{-\lambda_0 t}  \E \|{\bf Z}_0^{n,\vv}\|_1+C_0\Big(\frac{1}{n}\big(1+\E |X_0^1|\big)+\varphi(n)+   \vv \Big)\bigg)\\
  &= C_0\e^{-\lambda_0 t} \E|X_0^{1}-\bar X_0^{1}|+C_0\Big(\frac{1}{n}\big(1+\E |X_0^1|\big)+\varphi(n) \Big).
\end{align*}
Whence, combining with \eqref{EY}, it holds that
\begin{align*}
\limsup_{n\to\8}\Gamma_2(t,n)
  \le C_0\e^{-\lambda_0 t} \mathbb W_1(\mu,\bar{\mu}).
\end{align*}

Based on the previous  estimates on   $(\Gamma_i(t,n))_{1\le i\le 4}$,    the proof of   Theorem \ref{thm1} can be done.
\end{proof}

\noindent {\bf Acknowledgements.}\,\,
The research of Jianhai Bao is supported by the National Key R\&D Program of China (2022YFA1006004) and the  NSF of China (No. 12531007).
The research of Jian Wang is supported by the National Key R\&D Program of China (2022YFA1006003) and the  NSF of China (Nos. 12225104 and 12531007).


\begin{thebibliography}{99}

\bibitem{AR}
 Agram, N. and Rems, J.:  Deep learning for conditional McKean-Vlasov jump diffusions, {\it  Systems Control Lett.}, {\bf 201} (2025), paper no. 106100, 15 pp.


\bibitem{AO} Agram, N. and {\O}ksendal, B.:  Optimal stopping of conditional McKean-Vlasov jump diffusions, {\it  Systems Control Lett.}, {\bf 188} (2024), paper no. 105815, 6 pp.

\bibitem{AO2}
Agram, N. and  {\O}ksendal, B.:  Stochastic Fokker-Planck equations for conditional McKean-Vlasov jump diffusions and applications to optimal control, {\it SIAM J. Control Optim.}, {\bf 61} (2023),  1472--1493.


\bibitem{Aldous}  Aldous, D.:  Stopping times and tightness, {\it Ann. Probab.}, {\bf 6} (1978),   335--340.


\bibitem{BLW} Bao, J., Liu, Y. and Wang, J.: A note on L\'evy-driven McKean-Vlasov SDEs under monotonicity, arXiv:2412.01070.

\bibitem{BW}
Bao, J. and Wang, J.: Long time behavior of one-dimensional McKean-Vlasov SDEs with common noise, {\it J. Math. Anal. Appl.}, {\bf552}, (2025), paper no. 129819, 31 pp.


\bibitem{BLY} Bo, L.,  Li, T. and  Yu, X.: Centralized systemic risk control in the interbank system: weak formulation and Gamma-convergence, {\it Stochastic Process. Appl.}, {\bf 150} (2022), 622--654.

\bibitem{BWWY} Bo, L., Wang, J., Wei, X. and Yu, X.: Mean-field control with Poissonian common noise:
a pathwise compactiffcation approach, arXiv:2505.23441.

\bibitem{BWWY2}Bo, L., Wang, J., Wei, X. and Yu, X.:
Extended mean-field control problems with Poissonian common noise: stochastic maximum principle and
Hamiltonian-Jacobi-Bellman equation, arXiv:2407.05356.



\bibitem{BTV} Brown,  M., Trautmann,   S.T.,  Vlahu, R.:  Understanding bank-run contagion, {\it Manag.
Sci.}, {\bf 63} (2017),  2272--2282.


\bibitem{BLM}Buckdahn, R.,  Li, J. and  Ma, J.:  A general conditional McKean-Vlasov stochastic differential equation, {\it Ann. Appl. Probab.}, {\bf 33 } (2023),   2004--2023.



\bibitem{CD1} Carmona, R. and Delarue,     F.:  {\it Probabilistic Theory of Mean Field Games with Applications I: Mean Field FBSDEs, Control, and Games}, Springer International Publishing, Switzerland, 2018.


\bibitem{CD2} Carmona, R. and Delarue,    F.:  {\it Probabilistic Theory of Mean Field Games with Applications II: Mean Field Games with Common Noise and Master Equations}, Springer International Publishing, Switzerland, 2018.

\bibitem{Cava}
Cavallazzi, T.: Well-posedness and propagation of chaos for L\'{e}vy-driven McKean-Vlasov SDEs under
Lipschitz assumptions, arXiv:2301.08594.

\bibitem{DTM}
Delarue, F., Tanr\'e, E. and Maillet, R.: Ergodicity of some stochastic Fokker-Planck equations with additive common noise,  arXiv:2405.09950.


\bibitem{dES} dos Reis, G.,  Engelhardt, S. and Smith, G.:  Simulation of McKean-Vlasov SDEs with super-linear growth, {\it IMA J. Numer. Anal.}, {\bf 42} (2022),   874--922.








\bibitem{Eberle} Eberle, A.,  Guillin, A. and  Zimmer, R.:  Quantitative Harris-type theorems for diffusions and McKean-Vlasov processes,
{\it Trans. Amer. Math. Soc.}, {\bf 371} (2019),   7135--7173.

\bibitem{ELL}
Erny, X.,  L\"ocherbach, E. and Loukianova, D.:
Conditional propagation of chaos for mean field systems of interacting neurons,
{\it Electron. J. Probab.}, {\bf 26} (2021), paper no. 20, 25 pp.


\bibitem{GLM} Guillin, A., Le Bris, P. and Monmarch\'{e}, P.:  Convergence rates for the Vlasov-Fokker-Planck equation and uniform in time propagation of chaos in non convex cases, {\it Electron. J. Probab.}, {\bf 27} (2022), paper no. 124, 44 pp.






\bibitem{HSSb}
Hammersley, William R. P.,  \v{S}i\v{s}ka, D. and  Szpruch, L.:
Weak existence and uniqueness for McKean-Vlasov SDEs with common noise,
{\it Ann. Probab.}, {\bf 49} (2021),  527--555.

\bibitem{HR}
Hern\'{a}ndez-Hern\'{a}ndez, D. and Ricalde-Guerrero, J.H.: Conditional McKean-Vlasov differential
equations with common Poissonian noise: Propagation of chaos,  arXiv:2308.11564.

\bibitem{HR2}Hern\'{a}ndez-Hern\'{a}ndez, D. and Ricalde-Guerrero, J.H.: Mean-field games with common Poissonian noise: a maximum principle approach,  arXiv:2401.10952.




\bibitem{Huang}Huang, X.: Coupling by change of measure for conditional McKean-Vlasov SDEs and applications,
 {\it Stoch. Process. Appl.}, {\bf 179} (2025),  no. 104508, 17 pp.







\bibitem{Kac} Kac, M.:  Foundations of kinetic theory, in: {\it Proceedings of the Third Berkeley Symposium on Mathematical Statistics and Probability}, 1954--1955, vol. III, pp.\ 171–-197, Univ. California Press, Berkeley-Los Angeles, Calif., 1956.


\bibitem{KKLN}Khue Tran, N., Kieu, T.-T.,  Luong, Duc-T. Ngo, H.-L.,
On the infinite time horizon approximation for L\'{e}vy-driven McKean-Vlasov SDEs with non-globally Lipschitz continuous and super-linearly growth drift and diffusion coefficients,
{\it J. Math. Anal. Appl.}, {\bf 54} (2025), paper no. 128982, 38 pp.


\bibitem{KNRS}Kumar, C.,  Neelima,  Reisinger, C. and Stockinger, W.:
Well-posedness and tamed schemes for McKean-Vlasov equations with common noise, {\it Ann. Appl. Probab.}, {\bf 32} (2022),  3283--3330.

\bibitem{LSZ}
Lacker, D.,  Shkolnikov, M.,  Zhang, J.,  Superposition and mimicking theorems for conditional McKean-Vlasov equations, {\it J. Eur. Math. Soc. (JEMS)},  {\bf25} (2023),   3229--3288.



\bibitem{LMW} Liang, M.,  Majka, Mateusz B. and  Wang, J.: Exponential ergodicity for SDEs and McKean-Vlasov processes with L\'evy noise, {\it Ann. Inst. Henri Poincar\'e Probab. Stat.}, {\bf 57} (2021),  1665--1701.


\bibitem{LW}
Luo, D. and  Wang, J.:  Refined basic couplings and Wasserstein-type distances for SDEs with L\'{e}vy noises, {\it  Stochastic Process. Appl.}, {\bf 129} (2019),   3129--3173.



\bibitem{Maillet} Maillet, R.: A note on the long time behavior of stochastic McKean-Vlasov equations with common noise, arXiv:2306.16130.


\bibitem{MR}
Marinelli, C. and R\"ockner, M.: On maximal inequalities for purely discontinuous martingales in infinite dimensions, in: \emph{S\'eminaire de Probabilit\'es XLVI},
 Lecture Notes in Math., vol. \textbf{2123},
 Cham, Springer International Publishing,  2014, 293--315.



\bibitem{MH1}
McKean Jr, H.P.: A class of Markov processes associated with nonlinear parabolic equations, {\it Proceedings of the National Academy of Sciences}, \textbf{56} (1966), 1907--1911.



\bibitem{NBKG} Neelima, Biswas, S., Kumar, C., dos Reis, G. and Reisinger, C.: Well-posedness and tamed Euler schemes for McKean-Vlasov equations driven by L\'{e}vy noise, arXiv:2010.08585.



\bibitem{Pham}Pham, H.:
Linear quadratic optimal control of conditional McKean-Vlasov equation with random coefficients and applications,
{\it Probab. Uncertain. Quant. Risk}, {\bf1} (2016), paper no. 7, 26 pp.


\bibitem{Schuh} Schuh, K.: Global contractivity for Langevin dynamics with distribution-dependent forces and uniform in time propagation of chaos, 	{\it Ann. Inst. Henri Poincar\'e Probab. Stat.}, {\bf 60} (2024),   753--789.





\bibitem{STW}Shao, J., Tian, T. and Wang, S.:
Conditional McKean-Vlasov SDEs with jumps and Markovian regime-switching: wellposedness, propagation of chaos, averaging principle, {\it J. Math. Anal. Appl.}, {\bf 534} (2024),  paper no. 128080, 17 pp.


\bibitem{Sznitman} Sznitman, A.S.: \emph{Topics in Propagation of Chaos}, Springer, Berlin, 1991.




\bibitem{WR}Wang, F.-Y. and  Ren, P.: \emph{Distribution Dependent Stochastic Differential Equations},  World Scientific Publishing Co. Pte. Ltd., Hackensack, NJ,  2025.


\bibitem{Wanga}Wang, F.-Y.:  Exponential ergodicity for non-dissipative McKean-Vlasov SDEs, {\it  Bernoulli}, {\bf 29} (2023),  1035--1062.







\end{thebibliography}
\end{document}